\documentclass[12pt]{amsart}
\usepackage{amscd,graphicx,psfrag}
\usepackage{amssymb}
\usepackage{verbatim} % give multiline comments
\usepackage{color}
\usepackage{tikz-cd}
\usepackage[colorlinks,
    linkcolor={red!50!black},
    citecolor={blue!50!black},
    urlcolor={blue!80!black}]{hyperref}
\newcommand{\TryPackage}[3]{\IfFileExists{#1.sty}{\usepackage{#1}#2}{#3}}
\TryPackage{mathrsfs}{\renewcommand{\mathcal}{\mathscr}}{%
     \TryPackage{eucal}{}{}}
\newcommand*\wbar[1]{%        % widebar 
  \hbox{ \kern0.01em%
    \vbox{%
      \hrule height 0.5pt  % The actual bar
      \kern0.25ex%         % Distance between bar and symbol
      \hbox{%
        \kern-0.2em%       % Shortening on the left side
        \ensuremath{#1}%
        \kern-0.05em%       % Shortening on the right side
      }%
    }%
  \kern-0.03em}%
}

\newcommand{\lto}{\longrightarrow}
\newcommand{\wh}{\widehat}
\newcommand{\wt}{\widetilde}
\newcommand{\sm}{\smallsetminus}

\newcommand{\al}{\alpha}
\newcommand{\be}{\beta}
\newcommand{\ga}{\gamma}

\newcommand{\ep}{\varepsilon}

\newcommand{\la}{\lambda}
\renewcommand{\phi}{\varphi}
\newcommand{\si}{\sigma}

\newcommand{\Ga}{\Gamma}

\newcommand{\De}{\Delta}
\newcommand{\Si}{\Sigma}
\newcommand{\ZZ}{{\mathbb Z}}

\newcommand{\QQ}{{\mathbb Q}}
\newcommand{\RR}{{\mathbb R}}
\newcommand{\cE}{\mathcal E}

\newcommand{\cV}{\mathcal V}

\newcommand{\Aut}{\operatorname{Aut}}

\renewcommand{\int}{\operatorname{int}}
\newcommand{\lk}{\operatorname{\ell{\it k}}}
\newcommand{\width}{\operatorname{width}}

\newcommand{\VG}{\operatorname{\it VG}}
\newcommand{\WG}{\operatorname{\it WG}}

\newcommand{\VB}{\operatorname{\it VB}}

\newcommand{\EG}{\operatorname{\it EG}}
\newcommand{\QG}{\operatorname{\it QG}}
\newcommand{\RG}{\operatorname{\it \wbar G}}

\newtheorem{theorem}{Theorem}[section]
\newtheorem{lemma}[theorem]{Lemma}
\newtheorem{proposition}[theorem]{Proposition}

\newtheorem{corollary}[theorem]{Corollary}

\theoremstyle{definition}     % this and the line below gives definitions in roman
\newtheorem{definition}[theorem]{Definition}

\theoremstyle{remark}
\newtheorem{remark}[theorem]{Remark}
\newtheorem{example}[theorem]{Example}

\begin{document}

\baselineskip=17pt

\title[Virtual knot groups and almost classical knots]{Virtual knot groups and \\ almost classical knots}

\author[Boden]{Hans U. Boden}
\address{Mathematics \& Statistics, McMaster University, Hamilton, Ontario}
\email{boden@mcmaster.ca}

\author[Gaudreau]{Robin Gaudreau}
\address{Mathematics \& Statistics, McMaster University, Hamilton, Ontario}
\email{gaudreai@mcmaster.ca}

\author[Harper]{Eric Harper}
\address{Mathematics \& Statistics, McMaster University, Hamilton, Ontario}
\email{eharper@math.mcmaster.ca}

\author[Nicas]{Andrew J. Nicas}
\address{Mathematics \& Statistics, McMaster University, Hamilton, Ontario}
\email{nicas@mcmaster.ca}

\author[White]{Lindsay White}
\address{Mathematics \& Statistics, McMaster University, Hamilton, Ontario}
\email{whitela3@mcmaster.ca}

\thanks{
}


\date{September 15, 2016}
\begin{abstract}
We define  a group-valued invariant $\RG_K$ of virtual knots $K$ and show that $\VG_K=\RG_K *_\ZZ \; \ZZ^2,$ where $\VG_K$ denotes the virtual knot group introduced by Boden et al. We further show that $\RG_K$ is isomorphic to both the extended group $\EG_K$ of Silver--Williams and the quandle group $\QG_K$ of Manturov and Bardakov--Bellingeri.    

A virtual knot is called \emph{almost classical} if it admits a diagram with an Alexander numbering, and in that case we show that $\RG_K$ splits as $G_K*\ZZ$, where $G_K$ is the knot group. We establish a similar formula 
for mod $p$ almost classical knots and derive obstructions to $K$ being mod $p$ almost classical. 

Viewed as knots in thickened surfaces, almost classical knots correspond to those that are homologically trivial. We show they admit Seifert surfaces and relate their Alexander invariants to the homology of the associated infinite cyclic cover. We prove the first Alexander ideal is principal, recovering a result first proved by Nakamura et al.~using different methods. The resulting Alexander polynomial is shown to satisfy a skein relation, and its degree gives a lower bound for the Seifert genus. We tabulate almost classical knots up to 6 crossings and determine their Alexander polynomials and virtual genus.

\end{abstract} 

\maketitle

%%%%%%%%%%%%%%%%%%%%%%%%%%%%%%%%%%%%%%%%%%%%%%%%%%%%%%%%%%%%%%%%%%%%%%%%%
\section*{Introduction}
Given a virtual knot $K$, we examine a family of group-valued invariants of $K$ that includes the extended group $\EG_K$ introduced by Silver and Williams in \cite{SW-Crowell},
the quandle group $\QG_K$ that was first defined by Manturov \cite{M02} and later studied by Bardakov and Bellingeri \cite{BB}, and the virtual knot group $\VG_K$ introduced in \cite{VK}. We define yet another group-valued invariant denoted $\RG_K$ and called the reduced virtual knot group. Using this new group, we explore the relationships between $\EG_K, \QG_K$ and $\VG_K$ by proving that the virtual knot group splits as $\VG_K = \RG_K *_\ZZ \; \ZZ^2$, and that   $\EG_K \cong \RG_K \cong \QG_K$ are isomorphic. 

We then study the Alexander invariants of $\RG_K$, which by the previous results carry the same information about the virtual knot $K$ as $\EG_K, \QG_K$ and $\VG_K$. In \cite{VK}, the virtual Alexander polynomial  $H_K(s,t,q)$ is defined in terms of the zeroth order elementary ideal of $\VG_K$; and \cite{VK} shows that $H_K(s,t,q)$ admits a normalization, satisfies a skein formula, and carries useful information about the virtual crossing number $v(K)$ of $K$. The same is therefore true of the polynomial invariant $\wbar H_K(t,v)$ defined here in terms of the zeroth order elementary ideal of $\RG_K$. 

We use the group-valued invariant $\RG_K$ to investigate virtual knots that admit Alexander numberings, also known as \emph{almost classical knots}. We show that if $K$ is almost classical, then $\RG_K$ splits as $G_K * \ZZ$, and we prove a similar result for virtual knots that admit a mod $p$ Alexander numbering and we refer to them as mod $p$ almost classical knots. 
Thus, the group $\RG_K$ provides a useful obstruction for a knot $K$ to be mod $p$ almost classical, for instance the polynomial $\wbar H_K(t, v)$ must vanish upon setting $v  = \zeta$ for any non-trivial $p$-th root of unity $\zeta.$ 

Virtual knots were introduced by Kauffman in \cite{KVKT}, and they can be viewed as knots in thickened surfaces. From that point of view, almost classical knots correspond to knots which are homologically trivial. For an almost classical knot or link $K$, and we give a construction of an oriented surface $F$ with $\partial F = K$ which we call a Seifert surface for $K$. We use this surface to construct the infinite cyclic cover $X_\infty$ associated to $K$ and to give a new proof that the first elementary ideal $\cE_1$ is principal whenever $K$ is almost classical. 
 This was previously demonstrated by Nakamura, Nakanishi, Satoh, and Tomiyama in \cite{Nakamura-et-al} for almost classical knots, and their proof uses the Alexander numbering of $K$ to 
construct a nonzero element in the left kernel of an Alexander matrix associated to a presentation of $G_K$.
We define the Alexander polynomial $\De_K(t) \in \ZZ[t^{\pm 1}]$ for an almost classical knot or link $K$ in terms of $\cE_1$, and we show that it satisfies a skein formula, see Theorem \ref{Skein-Theorem}. Interestingly, Theorem 7 of \cite{Sawollek} implies that the skein formula does not extend to all virtual knots.   

There is an analogous result for mod $p$ almost classical knots $K$;  their first elementary ideal $\cE_1$ 
is principal over the ring $\ZZ[\zeta_p]$,
where $\zeta_p = e^{2 \pi i/p}$ is a primitive $p$-th root of unity.
Thus the Alexander ``polynomial'' $\De_K(\zeta_p)$ of a mod $p$ almost classical knot is defined, not as a Laurent polynomial, but rather as an element in $\ZZ[\zeta_p]$, the ring of integers in the cyclotomic  field $\QQ(\zeta_p)$. It is well-defined up to units in $\ZZ[\zeta_p]$.

Using Manturov's notion of parity \cite{M11}, we show how to regard mod $p$ almost classical knots in terms of an ascending filtration on the set of all virtual knots, and we prove that any minimal crossing diagram of a mod $p$ almost classical knot is mod $p$ Alexander numberable. 

Jeremy Green has classified virtual knots up to six crossings \cite{green}, and in this paper we use his notation in referring to specific virtual knots, for example in Figure \ref{AC-knots}  showing the Gauss diagrams of all 76 almost classical knots with six or fewer crossings as well as in Table \ref{acks2} which lists their Alexander polynomial $\De_K(t)$ and virtual genus $ g(\Si_K).$

%%%%%%%%%%%%%%%%%%%%%%%%%%%%%%%%%%%%%%%%%%%%%%%%%%%%%%%%%%%%%%%%%%%%%%%%%

\section{Virtual knots, Gauss diagrams, and knots in surfaces} \label{section1}
In this section, we recall three equivalent definitions of virtual knots. The first definition is in terms of a virtual knot or link diagram, which consists of an immersion of one or several circles in the plane with only double points, such that each double point is either classical (indicated by over- and under-crossings) or virtual (indicated by a circle). The diagram is oriented if every component has an orientation, and two oriented virtual link diagrams are \emph{virtually isotopic} if they can be related by planar isotopies and a series of \emph{generalized Reidemeister moves} ($r1$)--($r3$) and ($v1$)--($v4$) depicted in Figure \ref{VRM}. Virtual isotopy defines an equivalence relation on virtual link diagrams, and a virtual link is defined to be an equivalence class of virtual link diagrams under virtual isotopy.

\begin{figure}[ht]
\centering
\def\svgwidth{300pt}
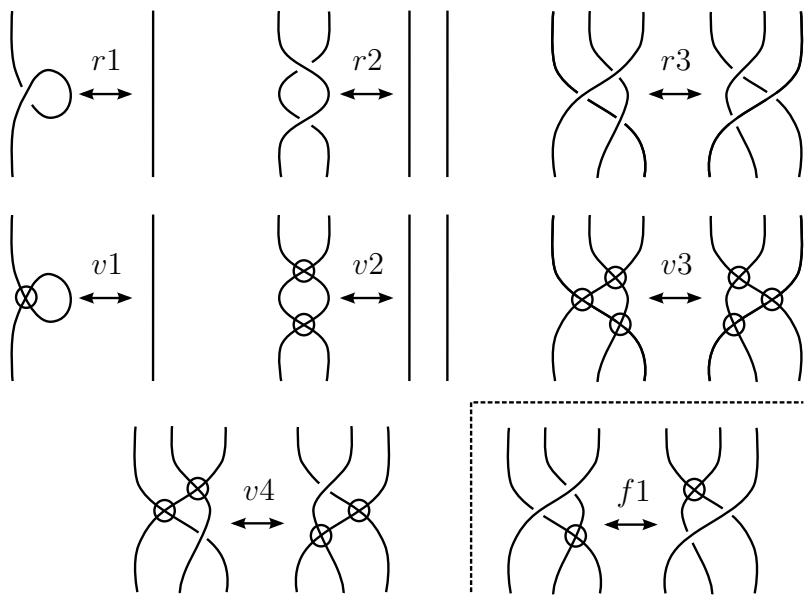
\caption{The generalized Reidemeister moves ($r1$)--($r3$) and ($v1$)--($v4$) and the forbidden overpass ($f1$).}
\label{VRM}
\end{figure}

One can alternatively define a virtual knot or link $K$ in terms of its underlying Gauss diagram, as originally proved by Goussarov, Polyak, and Viro \cite{GPV}. A Gauss diagram consists of one or several circles, one for each component of $K$, along with signed, directed chords from each over-crossing to the corresponding under-crossing. The signs on the chords indicate whether the crossing is right-handed ($+$) or left-handed ($-$). The Reidemeister moves can be translated into moves on Gauss diagrams, and two Gauss diagrams are called \emph{virtually isotopic} if they are related by a sequence of Reidemeister moves. Virtual isotopy defines an equivalence relation on Gauss diagrams, and virtual links can be defined as an equivalence classes of Gauss diagrams under virtual isotopy \cite{GPV}. 

\begin{figure}[ht]
\centering
\includegraphics[scale=0.90]{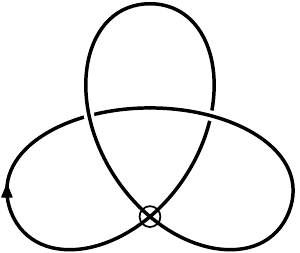} \qquad \qquad \includegraphics[scale=0.70]{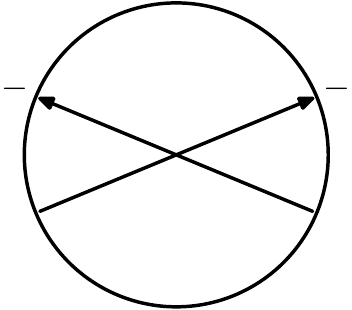}
\caption{The virtual trefoil and its Gauss diagram.}
\label{2-1}
\end{figure}

Given a Gauss diagram $D$ with chords $c_1, \ldots, c_n$, we define the index of the chord $c_i$ by counting the chords $c_j$ that intersect $c_i$ with sign and keeping track of direction.  Orient the diagram so that $c_i$ is vertical with its arrowhead oriented up. Then chords $c_j$ can intersect $c_i$ either from right to left or from left to right, and they can have sign $\ep_j = \pm 1.$  Counting them gives four numbers $r_+, r_-, \ell_+,$ and $\ell_-$ which are defined as follows:  
\begin{itemize}
\item[$r_\pm=$] number of $\pm$-chords intersecting $c_i$ with their arrowhead to the right,
\item[$\ell_\pm=$] number of $\pm$-chords intersecting $c_i$ with their arrowhead to the left.
\end{itemize}
%$r_-= \#\{\text{negative chords intersecting $c_i$ with their arrowhead to the right} \}$,
%\begin{itemize}  
%\item[$r_+$] $r_+= \#\{\text{negative chords intersecting $c_i$ with their arrowhead to the right} \}$,
%%is the number of positive chords intersecting $c_i$ with their arrowhead to the right,
%\item[$r_-$] is the number of negative chords intersecting $c_i$ with their arrowhead to the right,
%\item[$\ell_+$] is the number of positive chords intersecting $c_i$ with their arrowhead to the left, 
%\item[$\ell_-$] is the number of negative chords intersecting $c_i$ with their arrowhead to the left.
%\end{itemize}

%\begin{itemize} %\setlength\itemindent{-30pt}
%\item[] $r_+$ is the number of positive chords intersecting $c_i$ with their arrowhead to the right,
%\item[] $r_-$ is the number of negative chords intersecting $c_i$ with their arrowhead to the right,
%\item[] $\ell_+$ is the number of positive chords intersecting $c_i$ with their arrowhead to the left, 
%\item[] $\ell_-$ is the number of negative chords intersecting $c_i$ with their arrowhead to the left.
%\end{itemize}

\begin{definition} \label{chord-index}
In terms of these  numbers, 
the index of a chord $c_i$ in a Gauss diagram $D$ is defined to be  $I(c_i)=\ep_i(r_{+} - r_{-} + \ell_{-} -\ell_{+})$. 
\end{definition}
For example, the Gauss diagram in Figure \ref{2-1} has one chord with index $1$ and  another with index $-1$.

If a Gauss diagram $D$ is planar, then it is a consequence of the Jordan curve theorem that every chord $c_i$ has index $I(c_i)=0.$ This condition is necessary but not sufficient, the addition requirement is vanishing of the incidence matrix. For its definition as well as a full treatment of the planarity problem for Gauss words, see \cite{CE93, CE96}.

The third definition is in terms of knots on a thickened surface. Every virtual knot $K$ can be realized as a knot on a thickened surface, and by \cite{CKS} there is a one-to-one correspondence between virtual knots and stable equivalence classes of knots on thickened surfaces. Kuperberg proved that every stable equivalence class has a unique irreducible representative \cite{Kuperberg}, and the \emph{virtual genus} of $K$ is defined to be the minimum genus over all surfaces containing a representative of $K$. Thus, a virtual knot is classical if and only if it has virtual genus zero.
 
\begin{figure}[ht]
\centering
\includegraphics[scale=0.90]{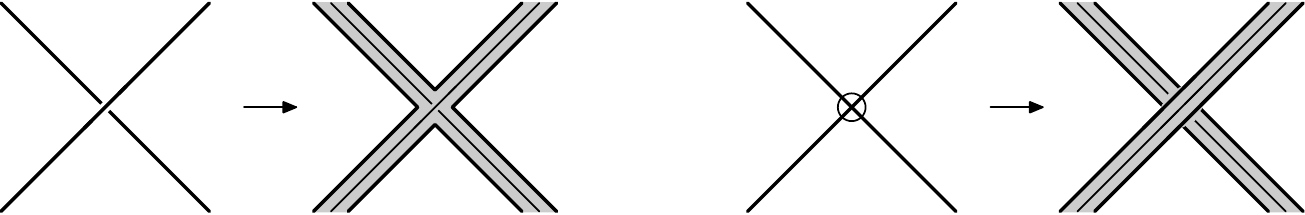}
\caption{The bands for a classical and virtual crossing.}
\label{band-surface}
\end{figure}

Given a virtual knot or link diagram $K$, there is a canonical surface $\Si_K$ called the Carter surface which contains a knot or link representing $K$. We review the construction of $\Si_K$, following the treatment of N. Kamada and S. Kamada in \cite{KK00}. (The reader may also want to consult Carter's paper \cite{Carter}.) The surface $\Si_K$ is constructed by attaching two intersecting bands at every classical crossing and two non-intersecting bands at every virtual crossing, see Figure \ref{band-surface}. Along the remaining arcs of $K$ attach non-intersecting and non-twisted bands, and the result is an oriented 2-manifold with boundary. By filling in each boundary component with a 2-disk, one obtains a closed oriented surface $\Si_K$ containing a representative of $K$.

Two virtual knots or links are said to be \emph{welded equivalent} if one can be obtained from the other by generalized Reidemeister moves plus the forbidden overpass ($f1$) in Figure \ref{VRM}. In terms of Gauss diagrams, this move corresponds to exchanging two adjacent arrow feet, see  Figure \ref{GD-forbidden}.

\begin{figure}[ht]
\centering
\includegraphics[scale=0.80]{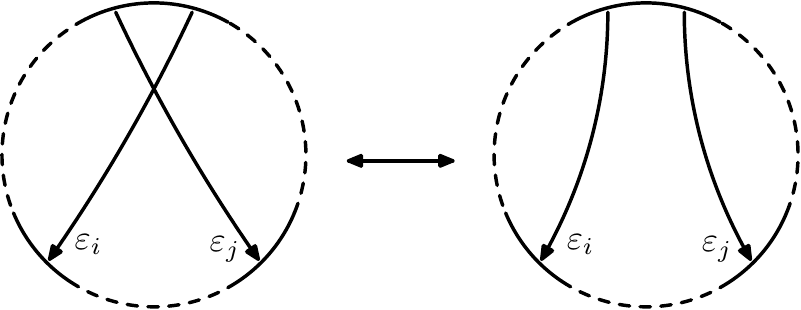}
\caption{The forbidden overpass $(f1)$ for Gauss diagrams.}
\label{GD-forbidden}
\end{figure}

%%%%%%%%%%%%%%%%%%%%%%%%%%%%%%%%%%%%%%%%%%%%%%%%%%%%%%%%%%%%%%%%%%%%%%%%%

\section{Group-valued invariants of virtual knots} \label{section2}
In this section, we introduce a family of group-valued invariants of virtual knots. We begin with the knot group $G_K$.

Suppose $K$ is an oriented virtual knot with $n$ classical crossings, and choose a basepoint on K. Starting at the base point, we label the arcs $a_1, a_2, \ldots, a_n$ so that at each undercrossing, $a_i$ is the incoming arc and $a_{i+1}$ is the outgoing arc. We use a consistent labeling of the crossings so that the $i$-th crossing is as shown in 
Figure \ref{Knot-Group-Wirtinger}.
For $i=1,\ldots, ,n$ let $\ep_i=\pm 1$ be according to the sign of the $i$-th crossing.
Then the {\it knot group of $K$} is the finitely presented group given by
$$G_K = \langle a_1, \ldots, a_n \mid a_{i+1} = a_j^{\ep_i} a_i a_j^{-\ep_i}, i=1, \ldots, n \rangle.$$
Note that virtual crossings are ignored in this construction.

\begin{figure}[ht]
\centering
\includegraphics[scale=0.90]{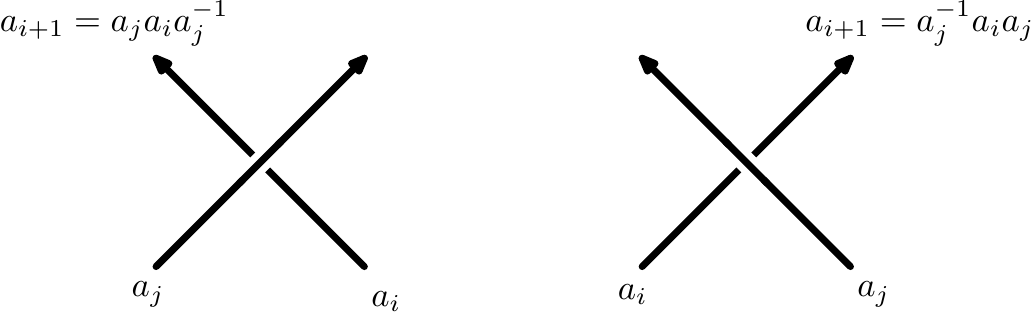}
\caption{The relations in $G_K$ from the $i$-th crossing of $K$.}
\label{Knot-Group-Wirtinger}
\end{figure}

The knot group $G_K$ is invariant under all the moves in Figure \ref{VRM} including the forbidden overpass ($f1$), thus it is an invariant of the underlying welded equivalence class of $K$. In case $K$ is classical, we have  $G_K \cong \pi_1(S^3 \sm N(K)),$ the fundamental group of the complement of $K$.

The virtual knot group $\VG_K$ was introduced in \cite{VK}; it has one generator for each short arc of $K$ and two commuting generators $s$ and $q$,  and there are two relations for each real and virtual crossing as in Figure \ref{VGK-Wirtinger}. 
Here the \emph{short arcs} of $K$ are arcs that start at one real or virtual crossing and end at the next real or virtual crossing.
\begin{figure}[ht]
\centering
\includegraphics[scale=0.90]{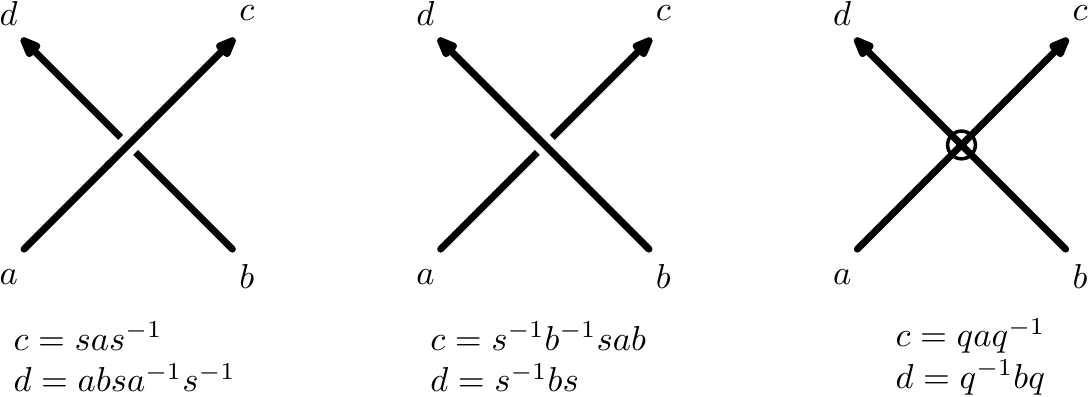}
\caption{The crossing relations for the virtual knot group $\VG_K$.}
\label{VGK-Wirtinger}
\end{figure} 

The virtual knot group $\VG_K$ has as quotients various other knot groups that arise naturally in virtual knot theory.  For instance, setting $q=s$, one obtains the welded knot group $\WG_K$, which is easily seen to be  invariant under welded equivalence (see \cite{VK}). By setting $s=1$, one obtains the quandle  group $\QG_K$, which was first introduced by Manturov in \cite{M02} and further studied by Bardakov and Bellingeri in \cite{BB}.

By setting $q=1$ one obtains the extended group $\EG_K$, which was introduced by Silver and Williams in \cite{SW-Crowell}, where it is denoted by $\wt{\pi}_K$. The extended group is closely related to the Alexander group, a countably presented group invariant of virtual knots that Silver and Williams use to define the generalized Alexander polynomial $G_K(s,t)$, see \cite{SW-Alexander}. 
 
The following diagram summarizes the relationship between $\VG_K, \QG_K,$ $\WG_K, \EG_K,$ and $G_K$.
\begin{center}
\begin{tikzcd}[column sep=large] 
&\VG_K 
\arrow{ldd}[swap]{s=1}  \arrow{dd}{q=s}  \arrow{rdd}{q=1} \\ \\
\QG_K   \arrow{rdd}[swap]{q=1} &   
\WG_K \arrow{dd}{s=1}  
&\EG_K \arrow{ldd}{s=1}  \\ \\ 
&G_K
\end{tikzcd} \\
{A commutative diamond for the augmented knot groups.}
\end{center}

Notice that the constructions of $G_K$ and $\EG_K$ make no reference to virtual crossings, and in fact for both knots the virtual crossing relations are trivial. 
Consequently, one can describe these groups entirely in terms of the Gauss diagram for $K$, which is advantageous in applying computer algorithms to perform algebraic computations. The constructions of $\QG_K$ and $\VG_K$ involve non-trivial virtual crossing relations, and so it is not immediately clear how to describe either of these groups in terms of Gauss diagrams. In the next section, we will introduce the reduced virtual knot group $\RG_K$,  and this group has the advantage of being computable from the Gauss diagram and we will see that it determines the virtual knot group $\VG_K$ and the quandle group $\QG_K$.

%%%%%%%%%%%%%%%%%%%%%%%%%%%%%%%%%%%%%%%%%%%%%%%%%%%%%%%%%%%%%%%%%%%%%%%%%

\section{The reduced virtual knot group} \label{section3}
In this section, we introduce the reduced knot group $\RG_K$ and 
show that the virtual knot group $\VG_K$ can be reconstituted from $\RG_K$. We then prove that $\RG_K$ is isomorphic  to both the quandle group $\QG_K$ and the extended knot $\EG_K$.  

\begin{figure}[h]
\centering
\includegraphics[scale=0.90]{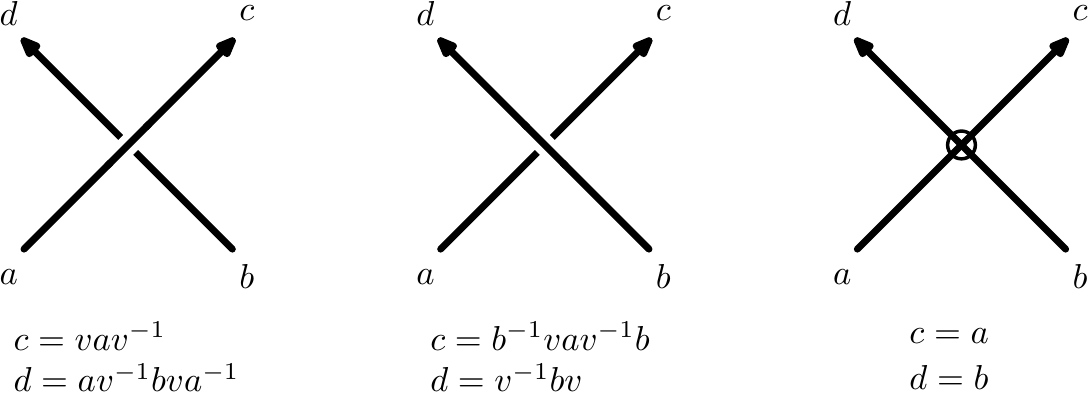}
\caption{The crossing relations for the reduced virtual knot group $\RG_K$.}
\label{Reduced-Wirtinger}
\end{figure}

Let $K$ be a virtual knot. 
The {\it reduced virtual knot group} $\RG_K$ has a Wirtinger presentation with generators given by the arcs of $K$ as labeled in Figure \ref{Reduced-Wirtinger} along with the augmentation generator $v$.  There are two relations for each classical crossing, as in Figure \ref{Reduced-Wirtinger}, and virtual crossings are ignored.

The first result in this section explains the relationship between the reduced virtual knot group $\RG_K$ and $\VG_K$. The proof will use the notion of an Alexander numbering, which we will define next.

Let $\Ga$ be a 4-valent oriented graph and 
let $S$ be the set of edges of $\Ga$.  We say that $\Ga$ has an {\it Alexander numbering} if there exists a function $\la \colon S \to \ZZ$ satisfying the relations in Figure \ref{Alexander-Numbering} at each vertex. A standard argument using winding numbers shows that any (classical) knot diagram admits an Alexander numbering. This was first observed by Alexander in \cite{Alexander}, and the argument uses the fact that an Alexander numbering of a knot diagram $\Ga$ is equivalent to a numbering of the regions of $\RR^2 \sm \Ga$ by the convention where a region $R$ has the number $\la_a$ for any edge $a$ with $R$ to its right, see Fig. 2 of \cite{Alexander} 

\begin{figure}[ht]
\centering
\includegraphics[scale=0.90]{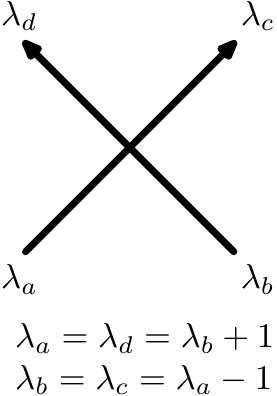}
\caption{The Alexander numbering of $\Ga$.}
\label{Alexander-Numbering}
\end{figure}

\begin{theorem}\label{vgk-amalgamation}
If $K$ is a virtual knot or link, then $\VG_K = \RG_K *_\ZZ\; \ZZ^2$.
\end{theorem}

\begin{proof}
To prove this result, we will make a change of variables that depends on an Alexander numbering of the underlying $4$-valent planar graph associated to $K$. 
Flatten $K$ by replacing classical and virtual crossings with $4$-valent vertices.  Denote the resulting planar graph as $\Ga$, and note that since $\Ga$ is planar, it admits an Alexander numbering.
(In Section \ref{section5},
we will introduce Alexander numberings for virtual knots, and we will see that not all virtual knots are Alexander numberable.)

We use the Alexander numbering to make a change of variables. Starting with the generators and crossing relations of $\VG_K$ from Figure \ref{VGK-Wirtinger}, make the substitution
\begin{equation} \label{AN-subst-1}
\left\{\begin{array}{lll}
 a' & = & q^{\la_a} \, a \, q^{-\la_a} \\
 b' & = & q^{\la_b }\, b \,q^{-\la_b } \\
 c' & = & q^{\la_c} \,c \,q^{-\la_c } \\
 d' & = & q^{\la_d} \,d \,q^{-\la_d}. 
\end{array} \right.
\end{equation}

In terms of the new generators $a', b', c', d'$, it follows that $\VG_K$ has crossing relations as in Figure \ref{VGK-2-Wirtinger}. 
\begin{figure}[ht]
\centering
\includegraphics[scale=0.90]{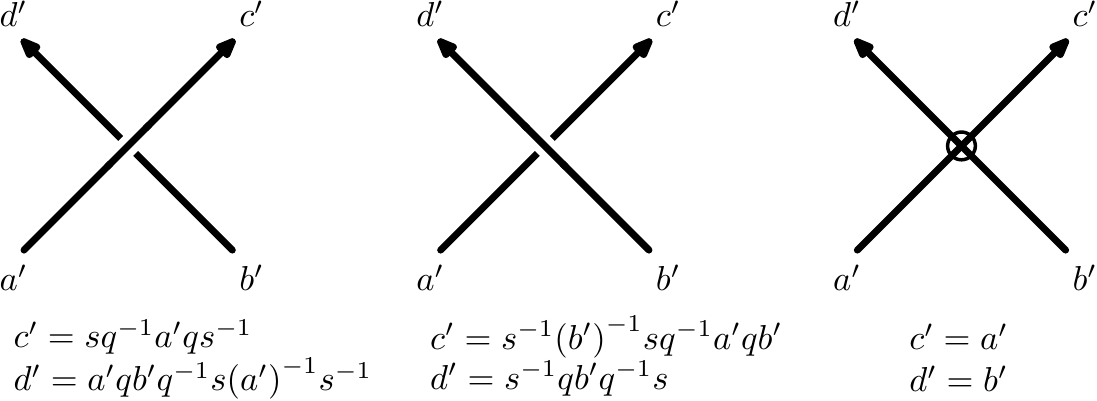}
\caption{The transformed crossing relations for  $\VG_K$.}
\label{VGK-2-Wirtinger}
\end{figure}

We make a further change of variables in $\VG_K$ by setting $a'' =  a' s,  b'' =  b' s,  c'' =  c' s,$ and $ d'' =  d' s$. Upon substituting $v$ for $sq^{-1}$ and noting that $s$ and $q$ commute in $\VG_K$, in terms of the new generators $a'',b'',c'',d''$, it is not difficult to check that the crossing relations for $\VG_K$  are identical to the relations for $\RG_K$ in Figure \ref{Reduced-Wirtinger}.
\end{proof}

\begin{remark}
A similar change of variables was used by Silver and Williams in their proof of \cite[Proposition 4.2]{SW-Crowell}. 
\end{remark}

In the next result, we use the reduced virtual knot group $\RG_K$  to construct isomorphisms between the quandle group $\QG_K$ and the extended group $\EG_K$.

\begin{theorem}\label{Duality}
If $K$ is a virtual knot or link, then the extended group $\EG_K$ and the quandle group $\QG_K$ are both isomorphic to the reduced virtual knot group $\RG_K$. 
\end{theorem}

\begin{proof}
We first argue that $\RG_K$ is isomorphic to $\EG_K$. To see this, we make the change of variables $ a' = a v^{-1},  b' = b v^{-1},  c' = c v^{-1},$ and $ d' = d v^{-1}$ in $\RG_K$.  Upon setting $v = s$, the crossing relations in the new variables $a',b',c',d'$ are identical with the crossing relations for $\EG_K$. 

To show that $\QG_K$ is isomorphic to $\RG_K$, we flatten $K$ by replacing classical and virtual crossings with $4$-valent vertices and denote the resulting planar graph as $\Ga$. Since $\Ga$ is planar, it admits an Alexander numbering which we use to perform a change of variables to the meridional generators of $\QG_K$ as in Equation \eqref{AN-subst-1}. 

In terms of the new generators $a',b',c',d'$, the crossing relations for $\QG_K$ are easily seen to be identical to the relations for $\RG_K$ in Figure \ref{Reduced-Wirtinger}. This completes the proof.
\end{proof}

The following diagram summarizes the results of Theorem \ref{vgk-amalgamation} and \ref{Duality}. The injective map $\RG_K \hookrightarrow \VG_K$ is induced by the map $\RG_K \, *_\ZZ \; \ZZ \hookrightarrow \RG_K \,*_\ZZ \; \ZZ^2 $ sending $v \mapsto sq^{-1}$, noting the isomorphisms $ \RG_K \, \cong \RG_K \, *_\ZZ \; \ZZ$ and 
$\VG_K \cong \RG_K \,*_\ZZ \; \ZZ^2.$

\begin{center}
\begin{tikzcd}[column sep=large] 
&\VG_K 
\arrow{ldd}[swap]{s=1}    
\arrow{rdd}{q=1} \\ \\
\QG_K \arrow[leftrightarrow]{r}{\cong}   \arrow{rdd}[swap]{q=1} &   
\RG_K \arrow[leftrightarrow]{r}{\cong} \arrow{dd}{v=1} \arrow[hookrightarrow]{uu}   
&\EG_K \arrow{ldd}{s=1}  \\ \\ 
&G_K
\end{tikzcd} \\
{A commutative diamond for the augmented knot groups.}
\end{center}

Because the crossing relations for the reduced group $\RG_K$ do not involve the virtual crossings, it follows that $\RG_K$ can be determined directly from a Gauss diagram for $K$. This has the following consequence.

\begin{corollary}
For any virtual knot or link $K$, the virtual knot group $\VG_K$ and the quandle group $\QG_K$ can be calculated from a Gauss diagram of $K$.
\end{corollary}

As we have already noted, the group obtained from $\VG_K$ upon setting $s=q$ is invariant under welded equivalence, cf. \cite{VK}. This group is denoted $\WG_K$ and is called the welded knot group. The next result shows that the welded knot group $\WG_K$ splits as a free product of the knot group $G_K$ and $\ZZ.$

\begin{proposition} \label{welded-group-decomp}
If $K$ is a welded knot or link, then $\WG_K = G_K * \ZZ$.
\end{proposition}

\begin{proof}
Let $K$ be virtual knot, then $\WG_K$ is the group generated by the short arcs of $K$ and one auxiliary variable $s$ with relations at each real or virtual crossing as in Figure \ref{VGK-Wirtinger} (setting $q=s$).

 We flatten $K$ by replacing classical and virtual crossings with $4$-valent vertices.  Denote the resulting graph as $\Ga$. Since $\Ga$ is planar, it admits an Alexander numbering, which we use  
to make the following change of variables:
 \begin{equation} \label{AN-subst-2}
\left\{\begin{array}{lll}
 a' & = & s^{\la_a} \, a \, s^{-\la_a} \\
 b' & = & s^{\la_b }\, b \,s^{-\la_b } \\
 c' & = & s^{\la_c} \,c \,s^{-\la_c } \\
 d' & = & s^{\la_d} \,d \,s^{-\la_d}. 
\end{array} \right. 
\end{equation}

In terms of the new generators $a',b',c',d',$
the crossing relations in $\WG_K$  transform into those in Figure \ref{Relations-2}.

\begin{figure}[ht]
\centering
\includegraphics[scale=0.90]{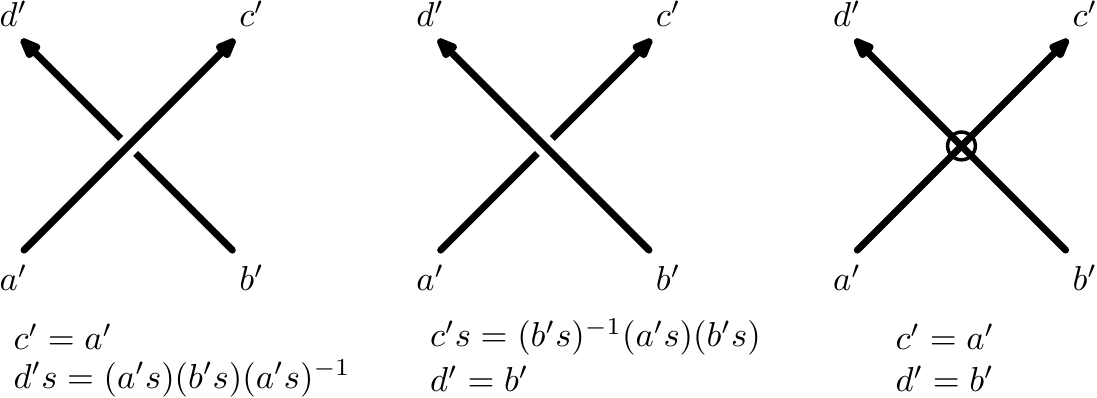}
\caption{The crossing relations for the welded knot group $\WG_{K}$}
\label{Relations-2}
\end{figure}  
We make a further change of variables by setting $a'' = a' s, b'' = b' s, c'' = c' s, d'' = d' s$, and in the new generators, one can easily see that the crossing relations are identical to those of $G_K$.
\end{proof}

We conclude this section by outlining an alternate approach to proving Theorem \ref{vgk-amalgamation} and Proposition \ref{welded-group-decomp} that involves the virtual braid group.  

The {\it virtual braid group} on $k$ strands, $\VB_k$, is defined by generators $\si_{1}, \ldots, \si_{k-1}$ and
$\tau_1,\ldots, \tau_{k-1}$ subject to the relations:
\begin{equation*}
\begin{array}{rcl}
\si_{i}\si_{j}&=&\si_{j}\si_{i} \hspace{1.5cm} \text{ if $|i-j|>1$,} \\
\si_{i}\si_{i+1}\si_{i}&=&\si_{i+1}\si_{i}\si_{i+1},\\ \\
 \tau_{i}\tau_{j}&=&\tau_{j}\tau_{i} \hspace{1.5cm} \text{ if $|i-j|>1$,} \\
\tau_{i}\tau_{i+1}\tau_{i} &=& \tau_{i+1}\tau_{i}\tau_{i+1}, \\
\tau_{i}^{2}&=&1,\\ \\
 \si_{i}\tau_{j}&=&\tau_{j}\si_{i} \hspace{1.5cm} \text{ if $|i-j|>1$,} \\
\tau_i \si_{i+1}\tau_i & = & \tau_{i+1} \si_i \tau_{i+1}.
\end{array}
\end{equation*} 

The generator $\si_i$ is represented by a braid in which the $i$-th strand crosses over the $(i+1)$-st strand and the generator $\tau_i$ is represented by a braid in which the $i$-th strand virtually crosses the $(i+1)$-st strand.

We recall the {\it fundamental representation} of $\VB_k$ from \cite[\S 4]{VK}.
For each $i=1,\ldots, k-1$, define automorphisms $\Phi(\si_i)$ and $\Phi(\tau_i)$ of
the free group $F_{k+2} = \langle x_1, \ldots, x_k, s, q\rangle$  fixing $s$ and $q$ as follows.
For $j=1,\ldots, k$, 
\begin{equation} \label{eq-fundrep}
\begin{split}
 (x_j) \Phi(\si_i) =&  \begin{cases} 
      s x_{i+1} s^{-1} &  \text{if $j=i$,} \\
      x_{i+1}x_{i} (s x_{i+1}^{-1}s^{-1}) & \text{if $j=i+1$,}\\
      x_{j} & \text{otherwise},
   \end{cases} \\
 (x_j) \Phi(\tau_i) =&  \begin{cases} 
      q x_{i+1} q^{-1} &  \text{if $j=i$,} \\
      q^{-1}x_{i}q &  \text{if $j=i+1$,} \ \\
      x_{j} & \text{otherwise}.
   \end{cases}
\end{split}
\end{equation}

Consider the following free basis for $F_{k+2}$:
\[
{\bar x_i} = \big(q^{i-1} x_i \, q^{-i +1}\big) q = q^{i-1} x_i \, q^{-i +2}  \text{ for }   1\leq i \leq k, \,
v = sq^{-1},  \text{ and }  s.
\]
Note we can use $q$ in place of $s$ for the last element of the basis.

A straightforward calculation gives the action of $\Phi(\si_i)$ and $\Phi(\tau_i)$ on the new basis, yielding:
\begin{equation} \label{eq-fundreptwo}
\begin{split}
 ({\bar x}_j) \Phi(\si_i) =&  \begin{cases} 
      v {\bar x}_{i+1} v^{-1} &  \text{if $j=i$,} \\
      {\bar x}_{i+1}{\bar x}_{i} (v {\bar x}_{i+1}^{-1}v^{-1}) & \text{if $j=i+1$,}\\
      {\bar x}_{j} & \text{otherwise},
   \end{cases} \\
 ({\bar x}_j) \Phi(\tau_i) =&  \begin{cases} 
      {\bar x}_{i+1}   &  \text{if $j=i$,} \\
      {\bar x}_{i}  &  \text{if $j=i+1$,} \ \\
      {\bar x}_{j} & \text{otherwise}.
   \end{cases}
\end{split}
\end{equation}
and $v$ and $s$ (and $q$) are fixed by $\Phi(\si_i)$ and $\Phi(\tau_i)$.

This calculation can be viewed as the non-abelian version of   \cite[Theorem 4.5]{VK}, that is, taking place in $\Aut(F_{k+2})$.

If $K= \wh{\be}$ is the closure of $\be \in \VB_k$, then the virtual knot group $\VG_K$ admits the presentation 
\begin{equation} \label{present3}
 \langle {\bar x}_1, \ldots, {\bar x}_k, v,s \mid  {\bar x}_1={\bar x}_1^\be, \ldots, {\bar x}_k={\bar x}_k^\be, [v,s]=1 \rangle.
\end{equation} 

Observe that $s$ does not appear in ${\bar x}_i^\be$ for $i=1,\ldots, k,$ and
 setting $s=1$ in (\ref{present3}) one obtains a presentation for the reduced virtual knot group $\RG_K,$ whereas setting $v=1$  in (\ref{present3}) one obtains a presentation for the welded knot group $\WG_K.$
It now follows easily from  (\ref{present3}) that  $\WG_K = G_K * \ZZ$  ($v$ goes to $1$) and $\VG_K = \RG \, *_\ZZ\; \ZZ^2$.

%%%%%%%%%%%%%%%%%%%%%%%%%%%%%%%%%%%%%%%%%%%%%%%%%%%%%%%%%%%%%%%%%%%%%%%%%

\section{Virtual Alexander invariants} \label{section4}
In this section, we introduce the Alexander invariants of the knot group $G_K$ and the reduced virtual knot group $\RG_K$.

We begin by recalling the construction of the Alexander invariants for the knot group $G_K = \langle a_1, \ldots, a_n \mid r_1, \ldots, r_n\rangle.$ Let $G_K' =[G_K,G_K]$ and $G_K'' = [G_K',G_K']$ be the first and second commutator subgroups, then the Alexander module is defined to be the quotient $G_K' / G_K''.$ It is finitely generated as a module over the ring $\ZZ[t^{\pm 1}]$, and it can be completely described in terms of a presentation matrix $A$, which is the $n \times n$ matrix obtained by Fox differentiating the relations $r_i$ of $G_K$ with respect to the generators $a_j$.
While the matrix $A$ will depend on the choice of presentation for $G_K$,  the associated sequence of elementary ideals
\begin{equation}\label{chain}
 (0)= \cE_0 \subset \cE_1 \subset \cdots \subset \cE_n  = \ZZ[t^{\pm 1}]
 \end{equation}
does not. Here, the $k$-th elementary ideal $\cE_k$ is defined as the ideal of $\ZZ[t^{\pm 1}]$ generated by all $(n-k) \times (n-k)$ minors of $A$. The Alexander invariants of $K$ are then defined in terms of the ideals \eqref{chain}, and since $G_K$ is an invariant of the welded type, these ideals are welded invariants of $K$.

For purely algebraic reasons, the zeroth elementary ideal $\cE_0$ is always trivial, and the reason is that 
  the fundamental identity of Fox derivatives implies that one column of $A$ can be written as a linear combination of the other columns. This is true for both classical and virtual knots, and it  is an algebraic defect of the knot group $G_K$. 
  
  Replacing $G_K$ by the reduced virtual knot group $\RG_K$, we will see that one obtains an interesting invariant from the zeroth elementary ideal of the associated Alexander module. 
In \cite{VK}, we introduce virtual Alexander invariants of $K$, which are defined in terms of the Alexander ideals of the virtual knot group $\VG_K$. Since $\VG_K$  abelianizes to $\ZZ^3$, its Alexander module is a module over $\ZZ[s^{\pm 1}, t^{\pm 1}, q^{\pm 1}]$. The zeroth order elementary ideal is typically non-trivial, and we define $H_K(s,t,q)$ to be the generator of the smallest principal ideal containing it. Thus $H_K(s,t,q)$ is an invariant of virtual knots, and in \cite{VK}, we show that $H_K(s,t,q)$ admits a normalization, satisfies a skein formula, and carries information about the virtual crossing number of $K$.   In \cite{VK}, we proved a formula relating  the virtual Alexander polynomial $H_K(s,t,q)$ to the generalized Alexander polynomial $G_K(s,t)$ defined by Sawollek \cite{Sawollek}, and in equivalent forms by Kauffman--Radford \cite{KR}, Manturov \cite{M02}, and Silver--Williams \cite{SW-Alexander, SW-03}.  Specifically, Corollary 4.8 of \cite{VK} shows that
$$
H_K(s,t,q) = G_K(sq^{-1}, tq).
$$

Theorem \ref{vgk-amalgamation} implies that the Alexander module of $\RG_K$ contains the same information as the  Alexander module of $\VG_K,$ and Theorem \ref{Duality} shows the same is true for the Alexander modules of the quandle group $\QG_K$ and the extended group $\EG_K$. We shall therefore focus our efforts on understanding the Alexander invariants associated to  $\RG_K \cong \QG_K \cong \EG_K$.

Let $\RG_K = \langle a_1, \ldots, a_n, v \mid r_1, \ldots, r_n \rangle$ be a presentation of the reduced virtual knot group. Its Alexander module is defined to be the quotient ${\RG^{\, \prime}_K}/\RG_K^{\, \prime \prime},$ where $\RG_K^{\, \prime}$ and $\RG_K^{\, \prime \prime}$ 
denote the first and second commutator subgroups of $\RG_K$.
It is finitely generated as a module over $\ZZ[t^{\pm t}, v^{\pm 1}]$, and one obtains a presentation matrix for it by Fox differentiating the relations of $\RG_K$ with respect to the generators.
Specifically, the $n \times (n+1)$ matrix $$
M=
\begin{bmatrix}
A & \left(\frac{\partial r}{\partial v}\right)
\end{bmatrix}
$$
is a presentation matrix for this module. Here the $(i,j)$-th entry of $A$ is given by $\left. \frac{\partial r_i}{\partial a_j}\right|_{a_1,\ldots,a_n=t}$  and $\left(\frac{\partial r}{\partial v}\right)$ is a column vector with $i$-th entry given by $\left. \frac{\partial r_i}{\partial v_{}}\right|_{a_1,\ldots,a_n=t}$.
Let 
\begin{equation}\label{bar-chain}
 (0) \subset \bar{\cE}_0 \subset \bar{\cE}_1 \subset \cdots \subset \bar{\cE}_n  = \ZZ[t^{\pm 1}]
 \end{equation}
denote the sequence of elementary ideals, where $\bar{\cE}_k$ is the ideal of $\ZZ[t^{\pm 1}, v^{\pm 1}]$ generated by the $(n-k) \times (n-k)$ minors of $M$.
The  $k$-th Alexander polynomial $\bar{\De}^\ell_K(t,v)$ is defined to be the generator for the smallest principal ideal containing $\bar{\cE}_k$; alternatively it is given by the $\gcd$ of the $(n-k) \times (n-k)$ minors of $M$. 

We use $\wbar{H}_K(t,v) =\bar{\De}^0_K(t,v)$ to denote the zeroth Alexander polynomial associated to $\RG_K$. Note that $\wbar{H}_K(t,v)$ is well-defined up to multiplication by units in $\ZZ[t^{\pm 1}, v^{\pm 1}]$
and is closely related to the virtual Alexander polynomial $H_K(s,t,q)$ and the generalized Alexander polynomial $G_K(s,t)$. In particular, one can define a normalization of $\wbar{H}_K(t,v)$ using braids by following the approach  in Section 5 of \cite{VK}, and one can also derive lower bounds on $v(K)$, the virtual crossing number of $K$, from $\wbar{H}_K(t,v)$ and its normalization.

\begin{theorem}
\label{GKpolyrelatedtoHbarpoly}
The generalized Alexander polynomial $G_K(s,t)$ is related to the reduced Alexander polynomial $\wbar{H}_K(t, v)$ by the formula
$$
G_K(s, t) = \wbar{H}_K(st, s).
$$
\end{theorem}

\begin{proof}
Choose a presentation for the reduced virtual knot group $\RG_K$ with relations as in Figure \ref{Reduced-Wirtinger}.  
Given a word $w$ in $a_1,\ldots, a_n, v$ and an automorphism $f \in \Aut(F_{n+1})$, we write $w^f$ for the word obtained by applying $f$ to $w$. Notice that this extends to give an automorphism of the group ring $\ZZ[F_{n+1}].$
If $r_i$ is a relator of $\RG_K$, then we see that $\tilde r_i = {r_i}^f$ is a relator of $\EG_K$ where $f$ is the automorphism that sends meridional generators $a_j \to a_jv$ and fixes the augmented generator $v$.  Then by the chain rule for Fox derivatives, see Equation $2.6$ in \cite{Fox1953}, we have that 
$$
\frac{\partial \tilde r_i}{\partial a_j} = \frac{\partial (r_i)^f}{\partial a_j} =  \left( \frac{\partial {r_i}}{\partial a_k}\right)^f 
\frac{\partial a_k^f}{\partial a_j}.
$$ 
After abelianizing, and multiplying by an appropriate unit,  it follows that 
$$
\left(\frac{\partial \tilde r_i}{\partial a_j}\right)(v, t) = \left(\frac{\partial {r_i}}{\partial a_j}\right)(tv, v).
$$
\end{proof}

In \cite{VK}, we define twisted virtual Alexander invariants $H^\varrho_K(s,t,q)$ associated to representations $\varrho \colon \VG_K \to GL_n(R)$. The twisted virtual Alexander polynomials admit a natural normalization, and one can derive bounds on $v(K)$, the virtual crossing number of $K$ from $H^\varrho(s,t,q)$ and its normalization.

Using the reduced knot group $\RG_K$ in place of $\VG_K$, one can similarly define twisted Alexander invariants $\wbar{H}^\varrho_K(t,v)$ associated to representations $\varrho \colon \RG_K \to GL_n(R)$ by following the approach in \cite[\S 7]{VK}. In particular, the twisted Alexander polynomial $\wbar{H}^\varrho_K(t,v)$ admits a normalization and carries information about the virtual crossing number $v(K)$ just as before, but it is easier to work with since $\RG_K$ admits a presentation with fewer generators and relations than $\VG_K$.

%%%%%%%%%%%%%%%%%%%%%%%%%%%%%%%%%%%%%%%%%%%%%%%%%%%%%%%%%%%%%%%%%%%%%%%%% 

\section{Almost classical knots}\label{section5}   

In \cite[Definition 4.3]{SW-Crowell}, Silver and Williams coined the term  ``almost classical knot" for any virtual knot $K$ represented by a diagram with an Alexander numbering, and in \cite[Definition 2.4]{carteretal} they use mod 2 almost classical knot to refer to any virtual knot represented by a diagram with a mod $2$ Alexander numbering. The next definition provides a natural extension of these ideas by introducing the notion of a mod $p$ almost classical knot.

\begin{definition} \label{Mod-p-Def}
\begin{enumerate}
\item[(i)] Given an integer $p \geq 0$, we say that a virtual knot diagram is  {\it mod $p$ Alexander numberable} if there exists an integer-valued function $\la$ on the set of short arcs satisfying the relations in Figure \ref{Alexander-Numbering-Def}. In case $p=0,$ we say the diagram is \emph{Alexander numberable}.
\item[(ii)] Given a virtual knot or link $K$, we say $K$ is \emph{mod $p$ almost classical} if it admits a virtual knot diagram that is mod $p$ Alexander numberable. In case $p=0,$ we say that $K$ is \emph{almost classical}.
\end{enumerate}
\end{definition}
Note that if a knot is almost classical, then it is mod $p$ almost classical for all $p$. Note also that when $p=1$ there is no condition on $K$, so we will assume $p \neq 1$ throughout.

By flattening classical and virtual crossings, as in Theorem \ref{Duality}, the resulting planar graph $\Ga$ always has an Alexander numbering. The condition that a virtual knot diagram have an Alexander numbering is very restrictive.  For instance, there are only four distinct non-trivial almost classical knots with up to four crossings, see Theorem $1.3$ of \cite{Nakamura-et-al}.

Definition \ref{Mod-p-Def} has natural interpretations for Gauss diagrams and for  knots in surfaces as we now explain.
For Gauss diagrams, this hinges on the fact that a Gauss diagram is mod $p$ Alexander numberable if and only if each chord $c_i$ has index $I(c_i) = 0$ mod $p$ (see Definition \ref{chord-index}). The verification of this fact is not difficult and is left as an exercise. 

For knots in surfaces, we claim that a knot $K$ in a thickened surface $\Si\times [0,1]$ is mod $p$ Alexander numberable if and only if it is homologically trivial as an element in $H_1(\Si; \ZZ/p),$ where $\ZZ/p$ denotes the cyclic group of order $p$. In case $p=2,$ the proof of this can be found in \cite[\S 3]{carteretal}. (That argument also explains why, for a virtual knot diagram, the notion of a mod 2 Alexander numbering and checkerboard coloring coincide.) 
In general, recall from Section \ref{section1} that given a virtual knot diagram $K$, its Carter surface $\Si_K$ is constructed by 
first thickening the classical and virtual crossings of $K$ as in Figure \ref{band-surface} and then adding finitely many 2-disks $D_1, \ldots, D_n$. Set $\Si =\Si_K$ and let $K$ be the associated knot in the thickened surface $\Si\times [0,1]$.

A mod $p$ Alexander numbering of $K$ induces a mod $p$ Alexander numbering of the 2-disks $D_1, \ldots, D_n$ as follows. The Carter surface is built as a CW complex with 1-skeleton given by the thickening of $K$ and with 2-cells given by $D_1,\ldots, D_n$. We number $D_i$ using the Alexander number of any edge $a$ of $K$ with  $D_i$ to its right. Let $\la_i$ denote the associated number, and the conditions of Figure \ref{Alexander-Numbering} guarantee that $\la_i$ is well-defined and independent of choice of edge. They further imply that for every edge of $K$, the Alexander number of the 2-disk to its right is exactly one more than the Alexander number of the 2-disk to its left. It follows that $\partial \left( \la_1 D_1 + \cdots + \la_n D_n\right) = K$ with coefficients in $\ZZ/p.$ Thus, if $K$ admits a mod $p$ Alexander numbering, then the associated knot in the surface is homologically trivial in $H_1(\Si; \ZZ/p).$ Conversely, if $K$ is a knot in $\Si \times [0,1]$ and is homologically trivial in $H_1(\Si; \ZZ/p)$, then that induces a mod $p$ Alexander numbering of the 2-disks $D_1, \ldots, D_n$ of the Carter surface $\Si_K$, which in turn gives a mod $p$ Alexander numbering of $K$.

\begin{figure}[ht]
\centering
\includegraphics[scale=0.90]{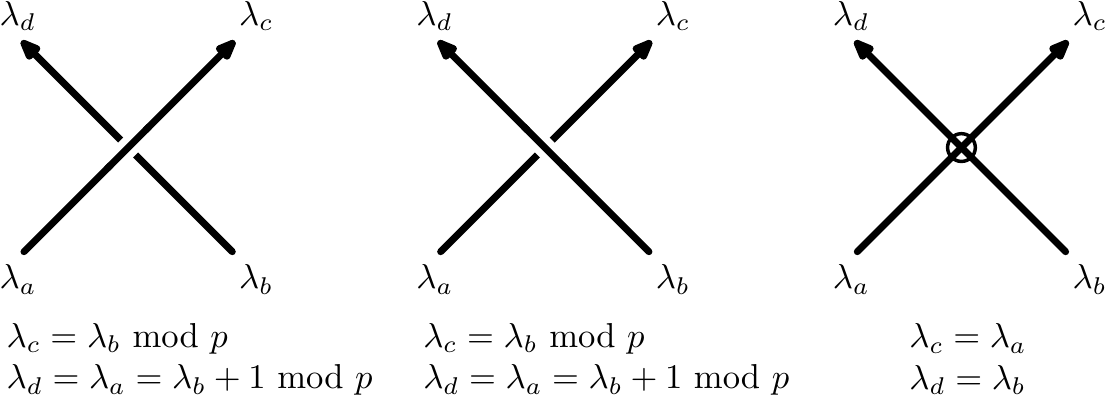}
\caption{The mod $p$ Alexander numbering conditions.}
\label{Alexander-Numbering-Def}
\end{figure}

The following theorem shows that $\RG_K$ gives an obstruction for a virtual knot or link to be  mod $p$ almost classical. 
\begin{theorem}\label{AC-Augmented-Theorem}
If $K$ is a mod $p$ almost classical knot or link, then it follows that $\RG_K / \langle v^p \rangle \cong G_K*\ZZ/p$. If $K$ is almost classical, then $\RG_K \cong G_K*\ZZ$.
\end{theorem}

\begin{proof}
Suppose $K$ is mod $p$ almost classical knot and choose a diagram for it that admits a mod $p$ Alexander numbering.  
Consider the quotient group $\RG_K/\langle v^p \rangle$ obtained by adding the relation $v^p =1$, and make the following change of variables to the generators of Figure \ref{Reduced-Wirtinger} according to the mod $p$ Alexander numbering of Figure \ref{Alexander-Numbering-Def}:
$$ \begin{array}{lll}
 a' & = & v^{\la_a} \, a \, v^{-\la_a} \\
 b' & = & v^{\la_b }\, b \, v^{-\la_b } \\
 c' & = & v^{\la_c} \, c \, v^{-\la_c } \\
 d' & = & v^{\la_d} \, d \, v^{-\la_d}. 
\end{array}$$
It is not difficult to check that, in the new generators $a',b',c',d',$ the crossing relations coincide with those for $G_K$.
\end{proof}

\begin{example} We will use Theorem \ref{AC-Augmented-Theorem} to see that the virtual knot $K = 4.41$ depicted in Figure \ref{4-41} is not mod 2 almost classical. Since this knot has $\wbar{H}_K(t, v)=0$, we are unable to conclude this from Corollary \ref{vanishatzeta} below.

Since $K$ is welded trivial, its knot group $G_K \cong \langle a \rangle$ is trivial. 
Using sage, we determine that $K$ has reduced virtual knot group 
\begin{eqnarray*}
\RG_K &=& \langle a, v \mid  v a^{-1}v^{-3} a^{-1} v^2 a^{-1} v^{-2} a v^3 a (v^ {-1} a^{-1})^2 \cdot \\
&& \hspace{1.5cm}  v^2 a v^{-2} (a v)^2 a^{-1} v^{-3} a^{-1} v^2 a v^{-2} a v^2 \rangle.
\end{eqnarray*}
Again using sage, we show that
$$\RG_K/\langle v^2 \rangle = \langle a, v \mid  v^2, \; v a^{-1} v a^{-1} (v a v a^{-1} v a)^2  \rangle.$$
One can show that this group admits precisely 18 representations into $S_3$, whereas the group $G_K * \ZZ/2 = \langle a, v \mid  v^2\rangle$ admits 24. This shows that $\RG_K/\langle v^2 \rangle$ is not isomorphic to $G_K * \ZZ/2$, and Theorem \ref{AC-Augmented-Theorem} implies that $K =4.41$ is not mod 2 almost classical.   
\end{example}

\begin{figure}[ht]
\centering
\includegraphics[scale=1.50]{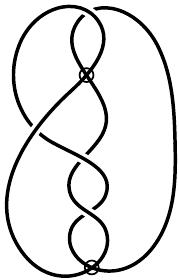}
\caption{The virtual knot $K=4.41$.}
\label{4-41}
\end{figure}

\begin{corollary}
\label{vanishatzeta}
If $K$ is a mod $p$ almost classical knot or link and $\zeta$ is a non-trivial $p$-th root of unity, then $\wbar{H}_K(t, \zeta) = 0$. If $K$ is almost classical, then $\wbar{H}_K(t, v)=0.$
\end{corollary}

\begin{remark}
The same result holds for twisted virtual Alexander polynomials $\wbar{H}^\varrho(t,v)$; if $K$ is a mod $p$ almost classical knot or link and $\varrho \colon \RG_K \to GL_n(R)$, then 
$\wbar{H}^\varrho(t,\zeta) = 0$ for any non-trivial $p$-th root of unity $\zeta$. If $K$ is almost classical, then $\wbar{H}^\varrho_K(t, v)=0.$
\end{remark}

\begin{proof}
Let $\RG_K = \langle a_1, \ldots, a_n, v \mid r_1, \ldots, r_n \rangle$. Then the $n \times (n+1)$ presentation matrix for the Alexander module of $\RG_K$ is given by 
$$
\begin{bmatrix}
A & \left(\frac{\partial r}{\partial v}\right)
\end{bmatrix},
$$
where $A= \left(\left. \frac{\partial r_i}{\partial a_j}\right|_{a_1,\ldots,a_n=t} \right)$ is the $n \times n$ meridional Alexander matrix and $\left(\frac{\partial r}{\partial v}\right)$ is the column vector whose $i$-th entry is $\left. \frac{\partial r_i}{\partial v_{}}\right|_{a_1,\ldots,a_n=t}$.
By the fundamental identity of Fox derivatives, it follows that $\wbar{H}_K(t,v) = \det A$.  

The result will be established by relating $\wbar{H}_K(t,v)$ to the first elementary ideal associated to the quotient group  $\RG_K / \langle v^p \rangle = \langle a_1, \ldots, a_n, v \mid r_1, \ldots, r_n, v^p  \rangle$. 
Its Alexander module has $(n+1)\times (n+1)$ presentation matrix given by  
$$
M=
\begin{bmatrix}
\,A & \left(\frac{\partial r}{\partial v}\right) \\[0.5em]
0 & \sum_{k=0}^{p-1} v^{k} 
\end{bmatrix}.
$$
Its first elementary ideal is the ideal generated by all the $n \times n$ minors of $M$, and in particular its Alexander polynomial
$\De^1(t,v)$ is given by the $\gcd$ of all these minors. Since $\det A$ is among the minors, it follows that $\De^1(t,v)$ divides $\det A = \wbar{H}_K(t,v)$.

If $K$ is mod $p$ almost classical, then by Theorem \ref{AC-Augmented-Theorem},
the quotient group admits a presentation of the form 
$$\RG_K / \langle v^p \rangle \cong G_K * \ZZ/p = \langle a_1, \ldots, a_n, v \mid r_1, \ldots, r_n, v^p \rangle,$$
where the relations $r_1, \ldots, r_n$ are independent of $v.$
The associated presentation matrix for the Alexander module is then given by  
$$
M'=
\begin{bmatrix}
\, A & 0 \\[0.3em]
0 & \sum_{k=0}^{p-1} v^{k} 
\end{bmatrix},
$$
where $A$ is now the Alexander matrix for $G_K$. The fundamental identity of Fox derivatives implies that $\det A =0$. Notice that the term $\sum_{k=0}^{p-1} v^{k}$ appears in each of the other $n\times n$ minors of $M',$ and this shows that $\sum_{k=0}^{p-1} v^{k}$ divides $\De^1(t,v)$, so it must also divide $\wbar{H}_K(t,v)$.  We conclude that  $\wbar{H}_K(t, \zeta) = 0$ if $\zeta$ is a non-trivial $p$-th root of unity. 
\end{proof}

Combining Theorem \ref{GKpolyrelatedtoHbarpoly}  and Corollary \ref{vanishatzeta} yields the following divisibility property of the generalized Alexander polynomial $G_K(s,t)$.

\begin{proposition}
If $K$ is a mod $p$ almost classical knot or link then the $p$-th cyclotomic polynomial $\Phi_p(s)$ divides $G_K(s,t)$.\qed
\end{proposition}

For example, consider the virtual knots $4.91$ and $4.92$ from the virtual knot table of Green \cite{green}.
The virtual knot $K=4.91$ has generalized Alexander polynomial 
$$G_K(s,t) = (s - 1)(t - 1)(st - 1)(s^2 t^2 + s t^2 + s^2 t + t^2 + 2st + s^2 + t + s + 1).$$
Since $G_K(s,t)$ is not divisible by any cyclotomic polynomial in $s$, we conclude that $K=4.91$ is not mod p almost classical for any $p$.

On the other hand, the virtual knot $K=4.92$ has generalized Alexander polynomial 
$$G_K(s,t) = (s - 1)(t - 1)(st - 1)(s^2 + s + 1)(t^2 + t + 1),$$ which is divisible by $s^2+s+1$. One can easily check that $K=4.92$ is mod 3 Alexander numberable.

%%%%%%%%%%%%%%%%%%%%%%%%%%%%%%%%%%%%%%%%%%%%%%%%%%%%%%%%%%%%%%%%%%%%

\section{Seifert surfaces for almost classical knots}\label{section6} 

In this section, we give a construction of the Seifert surface $F$ associated to an almost classical knot or link $K$. Our construction is modelled on Seifert's algorithm. In the following section, we will use the Seifert surface to show that the first elementary ideal $\cE_1$ is principal.

\begin{theorem} \label{thm-TFAE}
For a virtual knot or link $K$, the following are equivalent.
\begin{itemize}
\item[(a)] $K$ is almost classical, i.e., some virtual knot diagram for $K$ admits an Alexander numbering.
\item[(b)] $K$ is homologically trivial as a knot or link in $\Si \times [0,1]$, where $\Si$ is the Carter surface associated to an Alexander numberable diagram of $K$.
\item[(c)] $K$ is the boundary of a connected, oriented surface $F$ embedded in $\Si \times [0,1]$, where $\Si$ is again the Carter surface associated to an Alexander numberable diagram of $K$. 
\end{itemize}
The surface $F$ from part (c) is called a \emph{Seifert surface} for $K$. 
\end{theorem}

The argument that (a) $\Leftrightarrow$ (b) was given in Section \ref{section5}. If  $K = \partial F$ for an oriented surface $F \hookrightarrow \Si \times [0,1]$, then $K$ is necessarily homologically trivial. Thus (c) $\Rightarrow$ (b).
On the other hand, it is well known that an oriented, null homologous link in an oriented 3-manifold admits a Seifert surface (see \cite[Lemma 2.2]{CT} or \cite[Proposition 27.5]{Ranicki}). 
This section is devoted to providing an explicit algorithm for constructing the Seifert surface of an almost classical knot or link, and this will show that (b) $\Rightarrow$ (c) and complete the proof of the theorem.

We begin by recalling Seifert's algorithm. If $K$ is a classical knot, after smoothing all the crossings, we obtain a disjoint collection of oriented simple closed curves called Seifert circuits. Each Seifert circuit bounds a 2-disk in the plane, and the resulting system of 2-disks may be nested in $\RR^2$, but we can put them at different levels in $\RR^3$ and thus arrange them to be disjoint. 
The Seifert surface $F$ is then built from the disjoint union $D_1 \cup \cdots \cup D_n$ of Seifert disks by attaching half-twisted bands at each crossing.

A slight modification of this procedure enables us to construct Seifert surfaces for homologically trivial knots $K$ in $\Si \times [0,1]$. 
First smooth the crossings of $K$ to obtain a collection $\ga_1, \ldots, \ga_n$ of oriented disjoint simple closed curves in $\Si \times [0,1]$. Since $K$ is homologically trivial, it follows that $\sum_i [\ga_i] =0$ in $H_1(\Si; \ZZ)$.

\begin{proposition} \label{prop-subsurf}
If $\ga_1, \ldots, \ga_n$ are disjoint oriented simple closed curves in an oriented surface $\Si$ whose union is homologically trivial, then there exists a collection $S_1, \ldots, S_m$ of connected oriented subsurfaces of $\Si$ such that $$\bigcup_{j=1}^m \partial S_j = \bigcup_{i=1}^n \ga_i.$$
The orientation of each $S_j$ may or may not agree with the orientation of $\Si$, and the boundary orientation on $\partial S_j$ is the one induced by the outward normal first convention. 
\end{proposition}

We can assume that $S_j$ has nonempty boundary $\partial S_j$ for each $j$, though we do not assume that $\partial S_j$ is connected. The subsurfaces may have nonempty intersection $S_i \cap S_j$, but their boundaries $\partial S_1, \ldots, \partial S_m$ are necessarily disjoint (since $\ga_1, \ldots, \ga_n$ are).

The subsurfaces $S_j$ are the building blocks for the Seifert surface and in fact, given Proposition \ref{prop-subsurf} one can easily construct the Seifert surface $F$ for $K$ as follows. By placing $S_1, \ldots,  S_m$ at different levels in the thickened surface $\Si \times [0,1]$, we can arrange them to be disjoint. We then glue in small half-twisted bands at each crossing of $K$, and the result is an oriented surface $F$ embedded in $\Si \times [0,1]$ with $\partial F = K$.  

We now complete the argument by proving the proposition.

\begin{proof}
The proof is by induction on $n,$ and the base case $n=1$ amounts to the observation that a single curve $\ga_1$ is homologically trivial if and only if it is a separating curve on $\Si$. In fact, the Alexander numbering of $\ga_1$ induces an Alexander numbering of $\Si \sm \ga_1$ with the region to the left of $\ga_1$ having number $\la_1 - 1$ and the region to the right having number $\la_1$. Thus $\Si \sm \ga_1$ is disconnected, and one can take $S_1$ to be either of the oriented subsurfaces of $\Si$ bounding $\ga_1$. 
\begin{figure}[ht]
\centering
\includegraphics[scale=1.10]{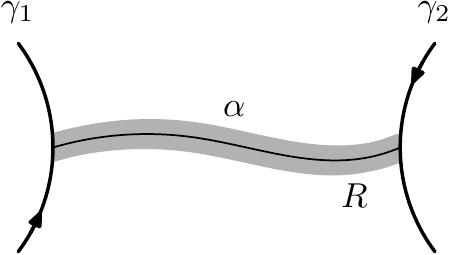}
\caption{Two curves $\ga_1$ and $\ga_2$, and a path $\al$ connecting them. The thin rectangle $R$ is the shaded region, and in this figure the orientations of $\ga_1$ and $\ga_2$ are consistent with that of $R$.}
\label{curves1}
\end{figure}

Now suppose $n=2.$ If $\ga_1$ is separating, then we again have  a subsurface $S_1$ with $\partial S_1 = \ga_1$, and we are reduced to the previous case. If instead $\ga_1$ and $\ga_2$ are both non-separating, we choose a path $\al \colon [0,1] \to \Si$ from $\ga_1$ to $\ga_2$. We can assume that, apart from its endpoints, the path $\alpha$ does not intersect $\ga_1$ or $\ga_2.$ Thicken $\al$ slightly to get a long thin rectangle $R$ embedded in $\Si$ with short sides along $\ga_1$ and $\ga_2.$ 
Orient $R$ so that its boundary orientation is consistent with $\ga_1$ on the short side along $\ga_1$. Then we claim that the orientation of $R$ on the other short side along $\ga_2$ is consistent with $\ga_2.$ We prove this by contradiction using the Alexander numbering. If the orientation is not consistent, then the regions of $\Si \sm (\ga_1 \cup \ga_2)$ would require three different Alexander numbers, implying that   $\Si \sm(\ga_1 \cup \ga_2)$ has at least three components, which contradicts the assumption that $\ga_1$ is non-separating.

\begin{figure}[ht]
\centering
\includegraphics[scale=1.10]{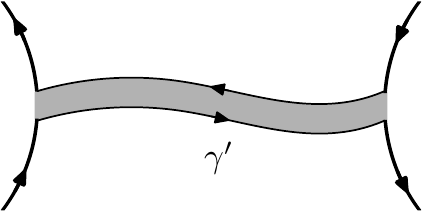}
\caption{The curve $\ga'$ is homologous to $\ga_1 \cup \ga_2$.}
\label{curves2}
\end{figure}

Thus, the orientations of $\ga_1$ and $\ga_2$ are both consistent with the rectangle $R$, and we replace $\ga_1 \cup \ga_2$ with a single curve $\ga'$ homologous to $\ga_1 \cup \ga_2$. Since $\ga'$ is null homologous, it is separating and that gives a subsurface $S$ of $\Si$ with $\partial S = \ga'$. We now perform surgery on $S$ to obtain the desired  subsurface $S'$ with $\partial S' = \ga_1 \cup \ga_2.$ Before describing the surgery operation, notice that due to boundary considerations, either $R$ is contained entirely in $S$ or they intersect only along the (long) edges of $R$.
The surgery operation amounts to removing $R$ from $S$ in the first case, or pasting $R$ into $S$ in the second. In the boundary, this amounts to replacing the long edges of $R$ with the short edges, and thus $\partial S' = \ga_1 \cup \ga_2.$

The inductive step follows by a similar argument. If any one of $\ga_1, \ldots, \ga_n$ is separating, then we have a subsurface $S_1$ bounding it and we can reduce to the case of $n-1$ null-homologous curves. On the other hand, if each $\ga_i$ is non-separating, then we claim we can find a pair of them, say $\ga_1$ and $\ga_2$, together with an arc $\al$ connecting them so that, the long thin rectangle $R$ we get from thickening $\al$ can be oriented consistently with $\ga_1$ and $\ga_2.$ Once that is established, we replace $\ga_1 \cup \ga_2$ by $\ga'$ as before and apply induction to the $n-1$ curves $\ga',\ga_3, \ldots, \ga_n$. 

It remains to be seen that the path $\al$ can be chosen so the orientations of $\ga_1$ and $\ga_2$ are consistent with the rectangle $R$ formed by thickening $\al$. To that end, choose a basepoint $p \in \ga_1$ and push-offs $p_+$ and $p_-$ in $\Si \sm \ga_1$ to the right and to the left of $p.$ Since $\ga_1$ is non-separating, we have a path $\al$ in $\Si \sm \ga_1$ connecting $p_+$ to $p_-$. Notice that the Alexander numbers of $p_+$ and $p_-$ are not equal. Notice further that the Alexander number of any point on $\al$ is constant except when $\al$ crosses one of the other curves $\ga_2, \ldots, \ga_n.$ We can assume $\al$ intersects each $\ga_i$ transversely. We can also assume, by relabeling, that the first curve $\al$ intersects is $\ga_2.$ If $\ga_1$ and $\ga_2$ are oriented consistently, then we can replace the pair $\ga_1 \cup \ga_2$ by $\ga'$ as before and apply induction to $\ga', \ga_3, \ldots, \ga_n$. Otherwise, we notice that the Alexander number of $\al$ must increase as it crosses $\ga_2.$  Continue in this manner, since the Alexander number of $p_-$ is one less than that of $p_+,$ eventually we must encounter a curve $\ga_i$ where the Alexander number of $\al$ decreases. One can verify that this curve and the previous one are consistently oriented, and using the segment of $\al$ connecting them one can form $\ga'$, reduce to the case of $n-1$ curves and apply induction. This produces a collection of subsurfaces whose boundary is the new set of curves, and the subsurfaces can again be modified by performing surgery, i.e. by either adding or removing the rectangle $R$, and this gives a collection of subsurfaces with boundary $\ga_1 \cup \cdots \cup \ga_n$.
\end{proof}

\begin{definition} The \emph{Seifert genus} of an almost classical knot or link $K$, denoted $g_s(K)$, is defined to be the minimum genus over all connected, oriented surfaces $F \hookrightarrow \Si \times [0,1]$ with $K = \partial F$ and over all Alexander numberable diagrams for $K$. 
\end{definition}

The next result shows that the genus of a Seifert surface is monotonically non-increasing under destabilization. 

\begin{theorem} \label{thm-genus}
Let $K$ be homologically trivial knot in a thickened surface $\Si \times I$ with Seifert surface $F$. 
Suppose $A$ is a vertical annulus in $\Si\times I$ disjoint from $K$, and let $\Si'$ be the surface obtained from $\Si$ by destabilization along $A$. (If destabilization along $A$ separates $\Si$, we choose $\Si'$ to be the component containing $K$.) 
Then there exists a Seifert surface $F'$ in $\Si' \times I$ for $K$ with $g(F') \leq g(F).$  
\end{theorem}

\begin{proof}
Given a vertical annulus $A$ for $\Si\times I$, consider its intersection
$F \cap A$ with the Seifert surface $F$. 
If $F \cap A$ is empty, then $F$ is a surface in $\Si' \times I$ and taking $F'=F$ satisfies the statement of the theorem. If on the other hand $F \cap A$ is non-empty, then it must
consist of a union of circles, possibly nested, in $A$. The idea is to perform surgery on the circles, one at a time, 
working from an innermost circle. (Innermost circles exist but are not necessarily unique.) 

Let $C$ denote such a circle. Viewing $C$ in $F$, then either it is a separating curve or a non-separating curve. If it's separating, then we can perform surgery to $F$ along $C$ in $\Si' \times I$ and 
 $F$ will split into a disjoint union $F' \cup F'',$ with
$F''$ a closed surface and with $\partial F' = K$. Clearly $F'$ is a surface in $\Si' \times I$ and $ g(F') \leq g(F)$ as claimed.

Otherwise, if the circle is non-separating in $F$, then using it to perform surgery to $F$, we obtain a Seifert surface $F'$ for $K$ in $\Si' \times I$ and $g(F') = g(F)-1.$
\end{proof}

Recall that for a virtual knot $K$, the \emph{virtual genus} of $K$ is the minimal genus over all thickened surfaces $\Si \times I$ in which $K$ can be realized. The next result is a direct consequence of Theorem \ref{thm-genus}, and 
it shows that a Seifert surface of minimal genus for an almost classical knot $K$ can always be constructed in the thickened surface $\Si \times I$ of minimal genus.  
\begin{corollary}
If $K$ is an almost classical knot with virtual genus $h$, then there exists a Seifert surface $F$ of minimal genus in the thickened surface $\Si \times I$ of genus $g(\Si) = h.$
\end{corollary}

%%%%%%%%%%%%%%%%%%%%%%%%%%%%%%%%%%%%%%%%%%%%%%%%%%%%%%%%%%%%%%%%%%%%%%%%%
\section{Alexander polynomials for almost classical knots}\label{section7} 
 
In this section,
we extend various classical results to the almost classical setting by using the Seifert surface constructed in the previous section. Recall that for a classical knot $K$ in $S^3$, the Alexander module can be viewed as the first homology of the infinite cyclic cover $X_\infty$, which can be constructed  by cutting the complement $S^3 \sm K$ along a Seifert surface and gluing countably many copies end-to-end (cf. \cite[Ch.~6]{Lickorish}).

We begin this section by performing a similar construction for almost classical knots $K$, and we use this approach to define Seifert forms $\al^\pm$ and Seifert matrices $V^\pm$, which we show give a presentation for the first elementary ideal $\cE_1$ of the Alexander module. It follows that $\cE_1$ is principal when $K$ is an almost classical knot or link, and this recovers and extends the principality result   in \cite{Nakamura-et-al}. We then define 
the Alexander polynomial  of an almost classical knot or link $K$ to be the polynomial $\De_K(t) =\det(tV^- -V^+) \in \ZZ[t^{\pm 1}]$ with $\cE_1 = (\De_K(t))$, which is well-defined up to multiplication by $t^g$ where $g$ is the virtual genus of $K$. 
(This is an improvement, for instance Lemma \ref{lem-standard} allows us to remove the sign ambiguity that one would have if one simply defined $\De_K(t)$ as a generator of the principal ideal $\cE_1$.) 

We give a bound on the Seifert genus $g_s(K)$ in terms of $\deg \De_K(t)$ and prove that $\De_K(t)$ is multiplicative under connected sum of almost classical knots. 
We establish a skein formula for $\De_K(t)$ and extend the knot determinant to the almost classical setting by taking $\det(K) = |\De_K(-1)|$, which we show is odd for almost classical knots. We briefly explore to what extent these results can be extended to mod $p$ almost classical knots and links. 

We begin by recalling the definition of {\it linking number}, $\lk(J,K)$ for a pair $(J, K)$ of disjoint oriented knots in $\Si \times I$.

\begin{proposition} \label{MVargument}
The relative homology group $H_1(\Si \times I \sm J, \, \Si \times 1)$ is infinite cyclic generated by a meridian $\mu$ of $J$.
\end{proposition}

\begin{proof}
We apply the Mayer-Vietoris sequence to the union $\Si \times I =  (\Si \times I \sm J) ~\cup~ \int(N(J)),$
where $N(J)$ is a closed regular neighborhood of $J$.
Let $j \colon \Si \to \Si \times I \sm J $ be given by $j(x) = (x,1)$,
let  $i \colon  \Si \times I \sm J \to \Si \times I$ be inclusion
and let $p \colon \Si \times I \to \Si$ be projection.
The composite
\[
\Si ~\stackrel{j}{\lto}~ \Si \times I \sm J ~\stackrel{i}{\lto}~ \Si \times I ~ \stackrel{p}{\lto} \Si
\]
is the identity and hence $i_*\colon H_*(\Si \times I \sm J) \to H_*(\Si \times I)$ is a split surjection
and $j_* \colon H_*(\Si) \to H_*(\Si \times I \sm J)$ is a split injection
and thus the Mayer-Vietoris sequence yields the short exact sequence
\[
0 \to H_1\left(\partial N(J) \right) ~\to~ H_1( \Si \times I \sm J ) \oplus H_1(J) ~\to~ H_1( \Si \times I) \to  0,
\]
moreover, $H_1\left(\partial N(J) \right) \cong \ZZ \oplus \ZZ$, generated by a meridian and longitude of $J$, and the summand generated by the longitude
maps isomorphically to $H_1(J) \cong \ZZ$.
The exact sequence of the pair  $\left(\Si \times I \sm J, \, \Si \times 1 \right)$  yields the split short exact sequence
\[
0 \to H_1\left(\Si \times 1 \right) ~\to~ H_1(\Si \times I \sm J) ~\to~ H_1(\Si \times I \sm J, \, \Si \times 1) \to  0.
\]
Combining the above two short exact sequences yields the conclusion.
\end{proof}

The knot $K$ determines a homology class $[K]$ in $H_1(\Si \times I \, \sm J, \, \Si \times 1)$
and  $\lk(J,K)$ is the unique integer $m$ such that $[K] = m \mu$.
A  geometric description of the linking number is given as follows.
Let $B$ be a $2$-chain in  $\Si \times I$ such that $\partial B = K - v$, where $v$ is a $1$-cycle in $\Si \times 1 \subset \Si \times I$.
Then  $\lk(J,K) = J \cdot B$ (the homological intersection number of $J$ with $B$), see \cite[\S1.2]{CT}.
Note that $ \lk(K,J)  - \lk(J,K)  = p_*([J]) \cdot p_*([K])$,  where $p_*([J]) \cdot p_*([K])$  is the intersection number in $\Si$ of the projections
of $J$ and $K$ to $\Si$,  \cite[\S1.2]{CT}, and so the linking number need not be symmetric.

Let $F$ be a compact, connected oriented surface of genus $g$ with $n$ boundary components, $n >0$, embedded in the interior of $\Si \times I$.
The homology groups $H_1(\Si \times I \, \sm F, \, \Si \times 1)$ and $H_1(F)$ are isomorphic, both are free abelian of rank $2g + n-1$, and,
furthermore, there is a unique non-singular bilinear form
\[
\beta \colon H_1(\Si \times I \, \sm F, \, \Si \times 1) ~\times~ H_1(F) ~\lto~ \ZZ
\]
such that $\beta([c],[d]) = \lk(c,d)$ for all oriented closed simple curves $c$ and $d$ in
$\Si \times I \sm F$ and $F$, respectively. 
The proof of this assertion is similar to the classical case (that is, $S^3$ in place of $\Si \times I$) as given in \cite[Proposition 6.3]{Lickorish}
and makes use of an analysis of the Mayer-Vietoris sequence of the union 
$\Si \times I  = (\Si \times I \sm F) ~\cup~ \int(N(F))$, where $N(F)$ is a closed regular neighborhood of $F$.

Assume that $F$ is a connected Seifert surface for a homologically trivial oriented link $L$ with $n$ components in $\Si \times I$.
The surface $F$ has a closed regular neighborhood $N(F)$ that can be given a
parametrization $F \times [-1,1] \cong N(F)$ with $F \times 0$ corresponding to $F$ and such that
the meridian of each component of $L$ enters $N(F)$ at $F \times -1$ and departs at $F \times 1$.
Let $\iota^\pm \colon F \to \Si \times I \sm F$ be the embeddings given by $\iota^\pm(x) = (x, \, \pm 1)$,
and we denote by $x^\pm = \iota^\pm(x)$ 
the ``positive and negative pushoffs''.
We thus obtain a pair of bilinear forms,  the {\it Seifert forms} of $F$,
\[
\al^\pm \colon H_1(F) ~\times~ H_1(F) ~\lto~ \ZZ
\]
given by $\al^\pm(x,y) = \beta(x^\pm, y)$.

Let $Z = \left( \Si \times I \, \sm L \right) / \, \Si \times 1$ be the space obtained from $\Si \times I \, \sm L $ by
collapsing \hbox{$\Si \times 1$} to a point.
The fundamental group of $Z$ is isomorphic to the link group $G_L$.
The abelianization of $G_L$ is
$H_1(G_L) \cong H_1(Z) \cong H_1(\Si \times I \, \sm L, \, \Si \times 1) \cong \ZZ^n$.
Let $\ep \colon \ZZ^n \to \ZZ$ be the homomorphism $\ep(a_1, \ldots, a_n) = \sum^n_{i=1}a_i$
and let $Z_\infty$ be the covering space of $Z$ corresponding to the homomorphism
$G_L \to H_1(G_L) \cong  \ZZ^n \stackrel{\ep}{\lto} \ZZ$.
We construct a model for $Z_\infty$, generalizing the well-known construction in \cite[{\S}6]{Lickorish},
using the Seifert surface $F$  together with the parametrization $F \times [-1,1] \cong N(F)$.
Let $X = \left( \Si \times I \, \sm \int(N(F)) \right) / \, \Si \times 1$.
Note that $X$ is a finite CW complex and $H_1(X) \cong H_1(\Si \times I \, \sm F, \, \Si \times 1)$.
Let $Y$ be the space obtained by cutting $X$ along $F$, that is, $Y$ is up to homeomorphism the space
$X \sm F$ compactified with two copies, $\iota^-(F)$ and $\iota^+(F)$, of $F$.
Let $\phi \colon \iota^-(F) \to \iota^+(F)$ be the homeomorphism determined by $\iota^\pm$.\
(Note that $X$ can be recovered from $Y$ by identifying $\iota^-(F)$ with $\iota^+(F)$ via $\phi$.) Set $Y_i = Y \times \{i\}$ for $i \in \ZZ$ and let $X_\infty$ be the space obtained from the disjoint union $\bigcup_{i \in \ZZ} Y_i$ by gluing $Y_i$ to $Y_{i+1}$ by identifying $F_i^- = \iota^-(F) \times \{i\}$ to $F_{i+1}^+ = \iota^+(F)\times \{i+1\}$ via $\varphi.$ In other words, $(y^-,i) \sim (\varphi(y)^+,i+1)$ for $y \in F$,
where $y^\pm = \iota^\pm(y)$ denote the positive and negative push-offs. Then $X_\infty$ is homeomorphic to the infinite cyclic cover $Z_\infty$,
and the homeomorphism   $t \colon X_\infty \to X_\infty$
given by $t(y,i) =(y,i+1)$
induces an isomorphism $t \colon H_1(X_\infty)  \lto H_1(X_\infty)$ and
gives $H_1(X_\infty)$ the structure of a $\ZZ[t^{\pm 1}]$-module.
The argument employed in \cite[Theorem 6.5]{Lickorish} extends to prove the following theorem.

\begin{theorem} \label{modulepresentation}
Let $F$ be a connected Seifert surface for a homologically trivial oriented link $L$ in $\Si \times I$.
Let $V^\pm$  be matrices for the Seifert forms $\alpha^\pm$
with respect to a given basis for $H_1(F)$.
Then $t\,V^- -\, V^+$ is a presentation matrix for the $\ZZ[t^{\pm 1}]$-module $H_1(X_\infty)$. \qed
\end{theorem}

\begin{corollary} \label{cor-AC}
The first elementary ideal of an almost classical knot or link $L$ is principal, generated by
$\det(tV^- - V^+)$ with $V^\pm$ as in Theorem \ref{modulepresentation}.
Hence, up to units,  $\De_L(t)$ coincides with $\det(t V^- - V^+)$.
\end{corollary}

\begin{remark}
Note that for a  classical link, the matrices $V^-$ and  $V^+$ are transposes of each other.   This need not be true if the genus of $\Si$ is positive.
Theorem \ref{modulepresentation} shows that our $\De_L(t)$ coincides (up to units) with
the {\it extended Alexander-Conway polynomial} of $L$ as defined by
Cimasoni and Turaev,
 \cite{CT}.
\end{remark}

The next result is standard, and for a proof see Lemma 6.7.6 of \cite{Cromwell}.
\begin{lemma} \label{lem-standard}
Suppose $L$ is an almost classical knot or link with $n$ components. Then
$$\det(V^- - V^+)= \begin{cases} 1 &\text{if $\, n=1$,}   \\
0 & \text{if $\, n>1.$ } \end{cases}$$
\end{lemma}

\begin{definition} \label{Alex-def}
If $K$ is an almost classical knot or link, we define its Alexander polynomial to be the Laurent polynomial $\De_K(t)=\det(t V^- - V^+)$. 
Theorem \ref{modulepresentation} shows that $\De_K(t)$ is a generator of the first elementary ideal $\cE_1,$ and Lemma \ref{lem-standard} shows that $\De_K(1)= 1$ for knots and that $\De_K(1)=0$ for links.
As an invariant of almost classical knots and links, $\De_K(t)$ is well-defined up to units in $\ZZ[t^{\pm 1}],$ and it is an invariant of the welded equivalence class of $K$.
\end{definition}

\begin{figure}[ht]
\centering
\includegraphics[scale=0.95]{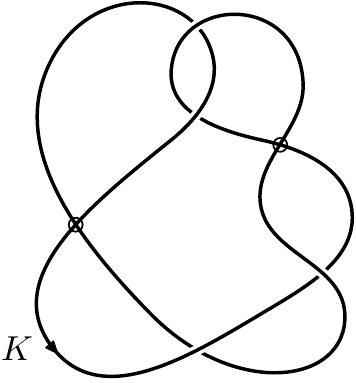} \qquad \qquad 
\includegraphics[scale=0.82]{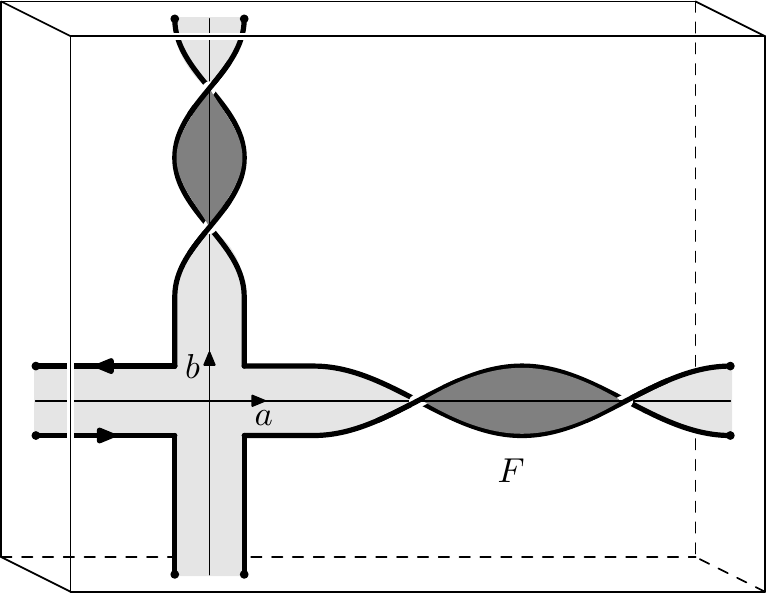}
\caption{The virtual knot $4.99$ on the left and its Seifert surface in $T^2 \times I$ on the right.} 
\label{Seifert}
\end{figure}    

\begin{example}
The Carter surface $\Si_K$ of the virtual knot $K=4.99$ has genus $g(\Si_K)=1$ and Figure \ref{Seifert} shows a virtual knot diagram for $K$ along with its Seifert surface $F \subset T^2 \times I.$ (In Figure \ref{Seifert}, $T^2 \times I$ is obtained from the cube on the right by identifying the two faces on top and bottom and the two faces on left and right in the obvious way. The front face is $T^2 \times 1$, and this choice determines how to compute linking numbers.)

This shows that $F$ is a union of two twisted bands and also has genus $g(F)=1.$   
We will use Figure \ref{Seifert} to compute the Seifert matrices $V^\pm$ of $F$ in terms of the generators $a,b \in H_1(F)$. We orient $F$ so that $a \cdot b =1$, and we orient $F \times [-1,1]$ so that $a^+$ and $b^+$ are obtained by pushing up along the darker region and down along the lighter region. Using the picture, one can see that
\begin{eqnarray*}
&\lk(a^+,a) = -1, & \lk(a^+,b) = 0, \\
&\lk(b^+,a) = 0, & \lk(b^+,b) = 1,
\end{eqnarray*}
and that implies $V^+ = \begin{bmatrix} -1 & 0 \\ 0 & 1 \end{bmatrix}.$
Similarly, we have
\begin{eqnarray*}
&\lk(a^-,a) = -1 & \lk(a^-,b) = 1 \\
&\lk(b^-,a) = -1 & \lk(b^-,b) = 1
\end{eqnarray*}
which implies $V^- = \begin{bmatrix} -1 & 1 \\ -1 & 1 \end{bmatrix}.$
Thus 
$$\De_K(t) = \det\left(t V^- -V^+\right) = \det\left(
\begin{bmatrix} 1-t & t \\ -t & t-1 \end{bmatrix}
\right)=2t-1.$$
\end{example}

\begin{remark}
Taking $\De_K(t)=\det\left(t^{1/2} V^- - t^{-1/2} V^+\right)$ gives a normalization (Conway polynomial) that is well-defined up to multiplication by $t^g$, where $g = g(\Si_K)$ is the genus of the minimal Carter surface. This follows by combining the argument on p.~560 of \cite{CT} with the statement of Kuperberg's theorem \cite{Kuperberg}. In the example above, since $g(\Si_K)=1,$ this means that $\De_K(t)$ has a well-defined sign determined by Lemma \ref{lem-standard}.
\end{remark}

For the Laurent polynomial $f(t)$, we define $\width(f(t))$ as the difference between its top degree and its lowest degree. The next result shows that the Seifert genus is bounded below by the width of $\De_K(t)$.

\begin{theorem} Suppose $K$ is an almost classical knot or link with $n$ components, then  $$\width\De_K(t) ~ \leq ~ 2g_s(K) + n -1.$$
\end{theorem}
\begin{proof} Given a Seifert surface $F$ for $K$, we see that $H_1(F)$ 
is free abelian of rank $2g(F) + n -1$.
Thus the Seifert matrices $V^+$ and $V^-$ are square matrices of size $2g(F) + n -1$,
and it follows that 
$$\width \De_K(t) \leq \deg \left(\det(t V^- - V^+)\right) \leq 2g(F) + n -1.$$
Taking a Seifert surface $F$ of minimal genus, we have $g_s(K) = g(F)$ and the theorem now follows.
 \end{proof}

The next result shows that the Alexander polynomial $\De_K(t)$ is multiplicative under connected sum. Note that although connected sum is not a well-defined procedure for virtual knots, it is nevertheless true that if $K_1$ and $K_2$ are both Alexander numberable diagrams, then so is $K_1 \# K_2$ for all possible choices, and the Alexander polynomial of $K_1 \# K_2$ is independent of these choices.

\begin{theorem} If $K_1$ and $K_2$ are both Alexander numberable, then so is $K_1 \# K_2$, and $\De_{K_1\# K_2}(t)= \De_{K_1}(t) \, \De_{K_2}(t).$ 
\end{theorem}

\begin{proof} Let $\Si_1$ and $\Si_2$ be the Carter surfaces for $K_1$ and $K_2,$ respectively, and construct Seifert surfaces $F_1 \subset \Si_1 \times I$
and $F_2 \subset \Si_2 \times I.$ Thus the knot $K_1\# K_2$ has Carter surface $\Si_1 \# \Si_2,$ and the boundary connected sum of $F_1$ and $F_2$ is a Seifert surface for $K_1\# K_2$. Theorem \ref{thm-TFAE} applies to show that $K_1 \# K_2$ is almost classical, and the Seifert matrices of the surface $F_1 \# F_2$ are easily seen to be block diagonal. Thus
$$V^\pm = \begin{bmatrix}  V^\pm_1 & 0 \\ 0 & V^\pm_2 \end{bmatrix},$$
where $V^\pm_1$ and $V^\pm_2$ denote the Seifert matrices of $F_1$ and $F_2,$ respectively. It follows that
\begin{eqnarray*}
\De_{K_1\# K_2}(t) &=& \det\left(t V^- -  V^+\right) \\
&=& \det\left(\begin{bmatrix}  tV^-_1-V^+_1 & 0 \\ 0 & tV^-_2 -V^+_2\end{bmatrix}\right) \\
&=& \det\left(tV^-_1-V^+_1\right) \, \det \left(tV^-_2-V^+_2\right) \\
&=& \De_{K_1}(t) \, \De_{K_2}(t),
\end{eqnarray*}
which proves the theorem.
\end{proof}

We will now show that $\De_K(t)$ satisfies a skein formula. 
Suppose that $K_+, \, K_-, \, K_0$ are three virtual knot diagrams which are identical outside a neighbourhood of one crossing, where they are related as in Figure \ref{skein}. Notice that if any one of $K_+, \, K_-, \, K_0$ is Alexander numberable, then so are the other two. This shows that skein theory makes sense for almost classical knots and links, and the next result gives a ``local" skein formula for the Alexander polynomial  $\De_K(t)$ for almost classical knots.
 
\begin{theorem}\label{Skein-Theorem}
Let $K_+, K_-$ and $K_0$ be three virtual knot diagrams that are identical outside of a small neighborhood of one crossing, where they are related as in Figure \ref{skein}.  We suppose that one (and hence all three) of the diagrams is Alexander numberable. Although each of $\De_{K_+}(t), \De_{K_-}(t)$ and $\De_{K_0}(t),$ is only well-defined up to multiplication by $\pm t^i$, if one uses the same Seifert surface for all three diagrams, then  $\De_{K_+}(t), \De_{K_-}(t)$ and $\De_{K_0}(t)$ become well-defined relative to one another. The Alexander polynomials satisfy the skein relation    
$$ \De_{K_+}(t) - \De_{K_-}(t) = (t - t^{-1}) \De_{K_0}(t).$$
\end{theorem}

\begin{proof} 
This result is proved by the same argument as for Lemma 3 in \cite{Giller}.
 \end{proof}

\begin{figure}[ht]
\centering
\includegraphics[scale=0.90]{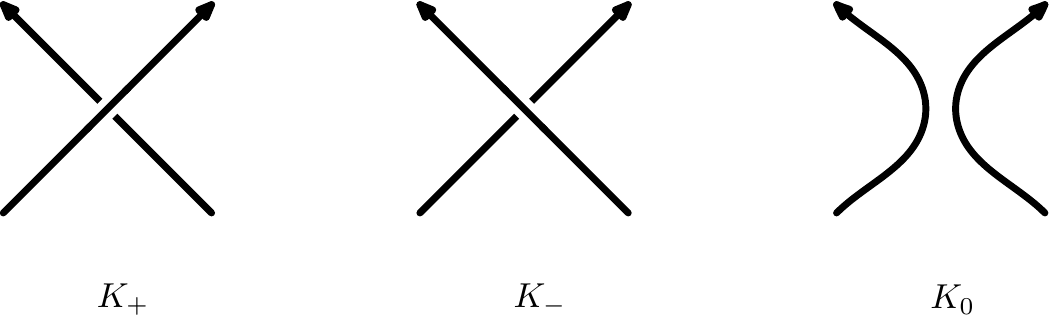}
\caption{The skein triple $K_+, \, K_-$ and $K_0.$} 
\label{skein}
\end{figure}    

\begin{lemma} \label{lem-odd}
If $K$ is an almost classical knot, then  $\De_K(-1)$ is odd. 
\end{lemma}
\begin{proof} The skein formula of Theorem \ref{Skein-Theorem} shows that $\De_K(-1) \mod 2$ is independent under crossing changes. However, by changing the crossings of $K$, we can alter the diagram so that $K$ is ascending. But any ascending virtual knot diagram is necessarily welded trivial, and any welded trivial knot has trivial Alexander polynomial. Thus $\De_K(-1) \equiv 1 \mod 2$ for any almost classical knot $K$, and this shows that $\De_K(-1)$ is odd.
\end{proof}

We now define the determinant of an almost classical knot or link. In case $K$ is an almost classical knot, Lemma \ref{lem-odd} shows that $\det(K)$ is odd.
In \cite{SW-06}, Silver and Williams define the determinant of a long virtual knot using Alexander group systems.

\begin{definition}
The \emph{determinant} of an almost classical link $L$ is given by
setting $\det(L)=| \De_L(-1)|$. Thus $\det(L) \in \ZZ$, and if $K$ an almost classical knot, then $\det(K)$ is  odd.
\end{definition}

\begin{remark}
In the classical case, the determinant of a link $L$ satisfies $\det(L) = | H_1(X_2)|,$ where $X_2$ denotes the double cover of $S^3$ branched along $L$, see Corollary 9.2 of \cite{Lickorish}. (Note that $\det(L)=0$ means $H_1(X_2)$ is infinite). For knots,  $\det(K)$ is odd thus $H_1(X_2)$ is always finite of odd order in that case. 

There is a similar story for almost classical knots and links. Suppose $L$ is an almost classical link, which we view as a link in $Z = \Si \times I/\Si \times 1$, the space obtained by collapsing one component of the boundary of $\Si \times I$ to a point. If $F$ is a Seifert surface for $L$ in $\Si \times I$ and $V^\pm$ are the two Seifert matrices, then one can show that $V^+ + V^-$ is a presentation matrix for $H_1(X_2)$, where $X_2$ denotes the double cover of $Z$ branched along $L$. In particular, it follows that $\det(L) = | H_1(X_2)|$, and  Lemma \ref{lem-odd} implies that $H_1(X_2)$ is finite of odd order in the case $L=K$ is an almost classical knot. 
\end{remark}

In  \cite{Nakamura-et-al} the authors use the Alexander numbering of $K$ to prove that the first elementary ideal $\cE_1$ is principal. Their proof extends to mod $p$ almost classical knots and links for $p \geq 2$ after making a change of rings which we now explain. Let $\zeta_p = e^{2\pi i/p}$ be a primitive $p$-th root of unity and consider
the change of rings $\phi\colon \ZZ[t^{\pm 1}] \to \ZZ[\zeta_p]$ given by setting  $\phi(t) = \zeta_p$.

Given a left $\ZZ[t^{\pm 1}]$-module $M$, we obtain the left $\ZZ[\zeta_p]$-module  $M'=  \ZZ[\zeta_p] \otimes_{\ZZ[t^{\pm 1}]} M$ where the right action of $\ZZ[t^{\pm 1}]$ on $\ZZ[\zeta_p]$ is given by $r\cdot a = r \phi(a)$
for $a \in \ZZ[t^{\pm 1}]$ and $r \in  \ZZ[\zeta_p]$.
Note that if $M$ is finitely presented and $\left(a_{ij} \right)$ is a presentation matrix for $M$ then $\left(\phi(a_{ij}) \right)$ is a presentation matrix for  $M'$.

Thus, if $K$ is a mod $p$ almost classical knot, $p \geq 2$, then we  define its Alexander ``polynomial'' to be an element $\De_K(\zeta_p) \in \ZZ[\zeta_p]$ with $\cE_1 = (\De_K(\zeta_p))$ over $\ZZ[\zeta_p]$.
As an invariant of mod $p$ almost classical knots, $\De_K(\zeta_p)$ is well-defined up to units in $\ZZ[\zeta_p],$ and it is an invariant of the welded equivalence class of $K$.  

The next proposition is the mod $p$ analogue of Corollary \ref{cor-AC}, and it follows by a straightforward generalization of the proof of Theorem 1.2 in \cite{Nakamura-et-al}. 
\begin{proposition} Suppose $p \geq 2$ and 
that $K$ is a mod $p$ almost classical knot or link. Then the first elementary ideal $\cE_1$ is principal over the ring $\ZZ[\zeta_p]$. \qed
\end{proposition}

%%%%%%%%%%%%%%%%%%%%%%%%%%%%%%%%%%%%%%%%%%%%%%%%%%%%%%%%%%%%%%%%%%%%%%%%%

\section{Parity and projection} \label{section8}
In this section, we introduce parity functions and the associated projection maps, and we explain their relationship with mod $p$ almost classical knots. Parity is a powerful tool with far-reaching implications, and we only give a brief account of it here. For more details, we refer the reader to Manuturov's original article \cite{M11}, his book \cite{MI-State}, and the monograph \cite{IMN-11}.

Given a virtual knot diagram, a parity is a function that assigns to each classical crossing a value in $\{0,1\}$ such that the following axioms hold: 

\begin{enumerate} 
\item In a Reidemeister one move, the parity of the crossing is even.
\item In a Reidemeister two move, the parities of the two crossings are either both even or both odd.
\item In a Reidemeister three move, the parities of the three crossings are unchanged. Further, the three crossings can be all even, all odd, or one even and two odd. (I.e. we exclude the case that one crossing is odd and two are even.) 
\end{enumerate}

Note that this notion is referred to as ``parity in the weak sense" by Manturov \cite{Manturov-12}.

For example, recall the definition of the index $I(c_i)$ of a chord in a Gauss diagram from Definition \ref{chord-index} and define
$$f(c_i)= \begin{cases} 0 & \text{if $I(c_i) \equiv 0$ mod $p$,}\\
1 & \text{if $I(c_i) \not\equiv 0$ mod $p$.}
\end{cases}$$
When $p=0$, then $f(c_i)$ is called the Gaussian parity function.
The proof that this is a parity follows from the properties of the index.
Observe that $I(c_i)=0$ for any chord involved in a Reidemeister one move, that $I(c_i)=-I(c_j)$ for any two chords involved in a Reidemeister two move, and that, for any three chords   involved in a Reidemeister three move,  $I(c_i),I(c_j),$ and $I(c_k)$ are unchanged  and satisfy $I(c_i)+I(c_j) = I(c_k)$. These observations imply that the function $f$ defined above satisfies the parity axioms.

Here we remark that in general there are many parity functions that can be defined on Gauss diagrams, but the parity function $f$ described above is ideally suited to studying mod $p$ almost classical knots, as we now explain. 

Given a virtual knot diagram $D$, we say that a chord $c$ is \emph{even} if $f(c) = 0$ and \emph{odd} if $f(c) \neq 0$. We will define a map  $$P_f \colon \{ \text{Gauss diagrams} \} \lto \{\text{Gauss diagrams}\}$$
called the \emph{Manturov projection map} as its definition was inspired by \cite{Manturov-12}
as follows.  If $D$ has only even chords, then $P_f(D)=D.$ Otherwise, if it has odd chords, then let $P_f(D)$ be the Gauss diagram $D\sm \{c \mid f(c) =1\}$ obtained by removing the odd chords. Note that there can be even chords in $D$ that become odd in $P_f(D).$  In terms of the virtual knot diagram, one removes odd crossings by virtualizing them.
Notice that a diagram $D$ admits a mod $p$ Alexander numbering if and only if $P_f(D) = D.$ Thus a virtual knot $K$ is mod $p$ almost classical precisely when it has a diagram that is fixed under $P_f.$

For example, Figure \ref{Man-proj} shows two Gauss diagrams. The one on the left has six chords, exactly two of which are odd with respect to the Gaussian parity. Under projection, it is sent to the almost classical knot $4.99$ whose Gauss diagram appears on the right.

\begin{figure}[ht]
 
\includegraphics[scale=0.70]{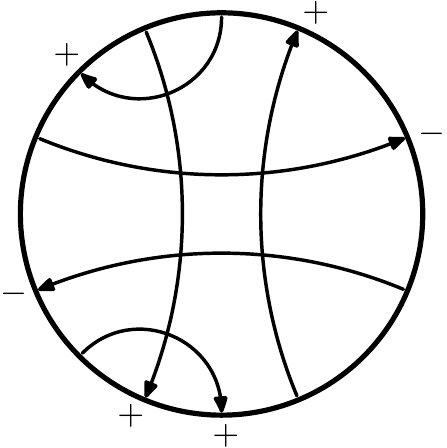} \qquad   \qquad
\includegraphics[scale=0.70]{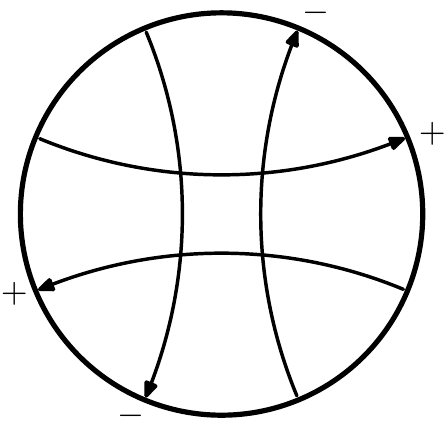} \hspace{0.2cm}  
\caption{Manturov projection $P_f$ sends the diagram on the left to the one on the right.} 
\label{Man-proj}
\end{figure}

The following lemma is crucial and implies that $P_f$ is well-defined as a map on virtual knots.

\begin{lemma}\label{Lemma-Projection}
The Manturov projection map respects Reidemeister equivalence, i.e., if $D$ is Reidemeister equivalent to $D'$, then $P_f(D)$ is Reidemeister equivalent to $P_f(D')$.
\end{lemma}

\begin{proof}
It is enough to prove the statement in the special case where $D$ and $D'$ are related by just one Reidemeister move.

Assume first that $D$ and $D'$ are related by a Reidemeister one move.  By the parity axioms, the chord involved in the Reidemeister one move is even. Further, we see that this chord move does not intersect any of the other chords of the Gauss diagram.  This means that $P_f(D)$ is identical to $P_f(D')$ outside of the local neighbourhood of the Reidemeister move. Thus $P_f(D)$ is related to $P_f(D')$ by a Reidemeister one move.  

Now assume that $D$ and $D'$ are related by a Reidemeister two move. Thus, two chords of opposite sign are added. Although the new chords may intersect other chords of the diagram, because they have opposite signs, they do not change the indices of any other chords. This means that $P_f(D)$ is identical to $P_f(D')$ outside of the local neighbourhood of the Reidemeister two move. 

The parity axioms imply that the two chords involved in the Reidemeister two move are either both even or both odd. If they are both even, then 
 $P_f(D)$ is related to $P_f(D')$ by a Reidemeister two move.
 If they are both odd, then we have
 $P_f(D)=P_f(D')$.

Finally, suppose $D$ and $D'$ are related by a Reidemeister three move.
Reidemeister three moves do not change the indices of any the chords in the diagram. By the parity axioms, either all three chords are even, or all three are odd, or exactly two of them are odd. If all three chords are even,  then $P_f(D)$ is related to $P_f(D')$ by a Reidemeister three move. If all three chords are odd or if exactly two chords are odd, then we have $P_f(D) = P_f(D').$ 
\end{proof}

The next result  implies that the set of mod $p$ almost classical knots and links form a self-contained knot theory under the generalized Reidemeister moves. It is a consequence of Lemma \ref{Lemma-Projection} and repeated application of the projection $P_f$. 

\begin{proposition}\label{AC-Projection}
Let $D$ and $D'$ be two Gauss diagrams that both admit mod $p$ Alexander numberings.  If $D$ is Reidemeister equivalent to $D'$, then they are related by a sequence of  Reidemeister moves through Gauss diagrams that admit mod $p$ Alexander numberings.
\end{proposition}

For a virtual knot $K$ represented by a diagram $D$, we let $P_f(K)$ be the virtual knot represented by the diagram $P_f(D)$. Lemma \ref{Lemma-Projection} implies that this map is well-defined as a map of virtual knots.
Notice further that $P_f(K) =K$ if and only if  $K$ is a mod $p$ almost classical knot.

Since $P_f^2 \neq P_f,$ it is not idempotent and so $P_f$ is not a projection in the strict sense.
However, for any given virtual knot $K$, there is an integer $n$ such that $P_f^{n+1}(K) = P_f^{n}(K).$ Thus, $P_f$ is eventually idempotent. 
In that case, it follows that $P_f^n(K)$ is mod $p$ almost classical, thus every virtual knot eventually projects to a mod $p$ almost classical knot under $P_f$. 
Thus, associated to every virtual knot $K$ is a canonical mod $p$ almost classical knot $K' = P_f^n(K),$ and using this association we see that any invariant of mod $p$ almost classical knots can be lifted to give an invariant of virtual knots.

For example, one can extend Definition \ref{Alex-def} to all virtual knots $K$ by defining $\De_K(t) := \De_{K'}(t),$
where $K'=P_f^n(K)$ is the image of $K$ under stable Manturov projection. (Here $f$ denotes the usual Gaussian parity, which implies that $K'$ is almost classical for $n$ sufficiently large.) 

Let $\cV$ denote the set of all virtual knots and for $n\ge 0,$ define 
$$\cV^p_n = \{ K \in \cV \mid P_f^n(K) \text{ is a mod $p$ almost classical knot} \}.$$
This gives an ascending filtration 
\begin{equation} \label{V-filtration}
\cV^p_0 \subset \cV^p_1 \subset \cdots \subset \cV^p_{n-1} \subset \cV^p_n \subset  \cdots  \cV,
\end{equation}
on the set $\cV$ of all virtual knots with the bottom $\cV^p_0$ consisting of the mod $p$ almost classical knots. We say that $K \in \cV$ lies in level $n$ if $K \in \cV_n^p \sm \cV^p_{n-1}$. 

For example, let $p=0$ and consider the projections of the knot $K=6.3555$, whose Gauss diagram is on the left of Figure \ref{V-step}. It has two odd chords, and removing them gives the Gauss diagram for 4.9, in the middle of Figure \ref{V-step}. This knot also has two odd chords, and removing them gives the Gauss diagram on the right of Figure \ref{V-step}, which we note is 2.1, hence non-trivial. Since the next projection $P_f^3(K)$ is trivial, it follows that $K \in \cV_3\sm \cV_2$. Thus $K$ lies in level 3. 

\begin{figure}[ht]
 
\includegraphics[scale=0.70]{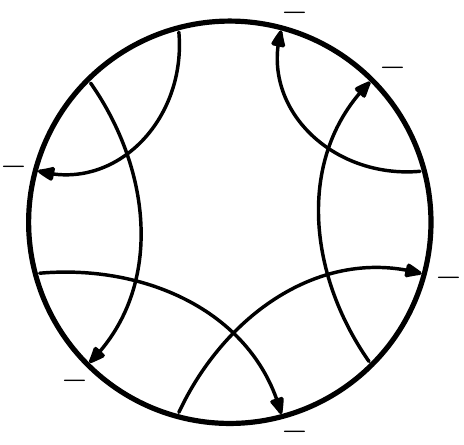} \qquad
\includegraphics[scale=0.70]{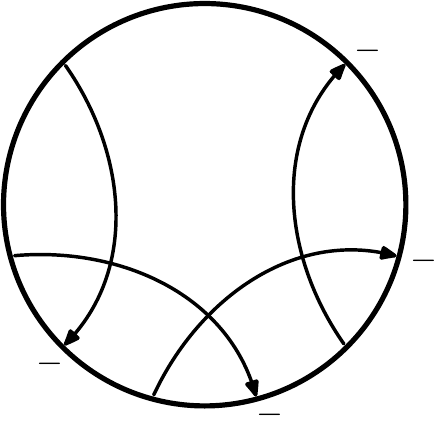} \qquad  
\includegraphics[scale=0.70]{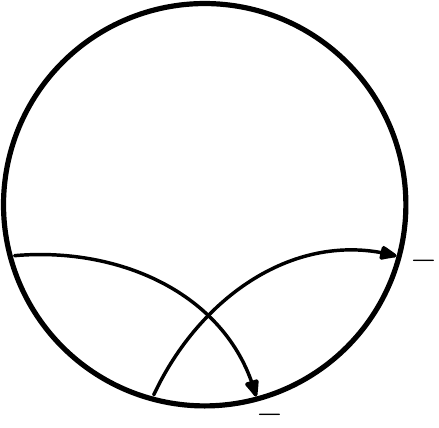} 
\caption{The Gauss diagrams of $K=6.3555$, $P_f(K)=4.9$, and $P_f^2(K)=2.1.$} 
\label{V-step}
\end{figure}

Manturov's projection map is especially useful for studying almost classical knots because it implies that any minimal crossing diagram for a mod $p$ almost classical knot
is mod $p$ Alexander numberable.
This observation is stated in the next theorem and is key to our tabulation of mod $p$ almost classical knots up to $6$ crossings in Section \ref{section9}.

\begin{theorem}\label{AC-Minimal}
If $K$ is a mod $p$ almost classical knot, then any minimal crossing diagram for $K$ is mod $p$ Alexander numberable.
\end{theorem}

\begin{proof} Let $K$ be a mod $p$ almost classical knot and suppose that $D$ is a minimal diagram for $K$.  Let $D'$ be a Gauss diagram of $K$ that has a mod $p$ Alexander numbering.  Since $D$ and $D'$ both represent the same knot, they are Reidemeister equivalent.   Since $D'$ has an mod $p$ Alexander numbering, $P_f(D')= D'$.  By Lemma \ref{Lemma-Projection}, $P_f(D)$ is Reidemeister equivalent to $P_f(D') = D'$.  Thus $P_f(D)$ is Reidemeister equivalent to $D$. Since $D$ is minimal, and since the projection map $P_f$ can only remove chords, it follows that $P_f(D) = D$.  Thus all chords of $D$ are even, and this implies that $D$ also admits a mod $p$ Alexander numbering. 
\end{proof}

For example, we apply this to Kishino's knot in Figure \ref{Kishino} with $p=2$ to see that no diagram for this knot admits a  mod 2 Alexander numbering. This resolves a question raised in Remark 3.11 of  \cite{carteretal}. 

\begin{figure}[ht]
\centering
\includegraphics[scale=1.60]{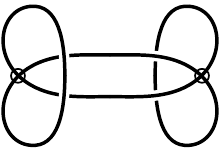}
\caption{The minimal crossing diagram for Kishino's knot is not mod $2$ Alexander numberable, thus no diagram of it is mod $2$ Alexander numberable.}
\label{Kishino}
\end{figure}

%%%%%%%%%%%%%%%%%%%%%%%%%%%%%%%%%%%%%%%%%%%%%%%%%%%%%%%%%%%%%%%%%%%%%%%%%

\section{Enumeration of almost classical knots}\label{section9}

In \cite{green}, Green has classified virtual knots up to six crossings. Since the virtual knots in his tabulation all have minimal crossing diagrams, by Proposition \ref{AC-Minimal}, to check whether a given virtual knot is mod $p$ almost classical, it is enough to check which diagrams in Green's table admit mod $p$ Alexander numberings. This is done by computer by computing the indices of all the chords in the Gauss diagram and seeing whether they are all equal to zero mod $p$. These calculations were performed using Matlab.  For brevity, Table \ref{Numbers-AC} includes only the numbers of mod $p$ almost classical knots for $p=0,2,3,4,5$, and we also give a complete enumeration of the  almost classical knots up to six crossings, along with computations of their Alexander polynomials $\De_K(t)$ and their virtual genus $g(\Si_K)$ in Table \ref{acks2}. Recall that $\Si_K$ is the Carter surface associated to $K$, and here we make use of Corollary 1 in \cite{Manturov-12}, which states that if $K$ is a minimal crossing diagram, then its Carter surface $\Si_K$ has minimal genus. The proof is similar in spirit to the proof of Theorem \ref{AC-Minimal}, though in the place of the parity function $f$ above, Manturov's argument uses homological parity.

Interestingly, every almost classical knot in Table \ref{acks2} with trivial Alexander polynomial also has trivial knot group, and a case-by-case argument reveals that each one of them is in fact welded trivial. 

Based on the low crossing examples, it is tempting to conjecture that almost classical knots $K$ have isomorphic upper and lower knot groups. We say that $K$ satisfies the \emph{twin group property} if $G_K$ and $G_{K^*}$ are isomorphic. Here $K^*$ is the vertical mirror image of $K$, which is the knot obtained from $K$ by switching all the real crossings. In terms of Gauss diagrams, it is obtained by reversing all the chords and changing their signs. 

Classical knots satisfy the twin group property, and this follows easily by interpreting the knot group $G_K$ as the fundamental group  $\pi_1(S^3 \sm \tau K)$ of the knot complement. Using the tabulation of almost classical knots, one can further check that the four-crossing and five-crossing almost classical knots also satisfy the twin group property.  However, there are counterexamples among the six-crossing almost classical knots. The first is the almost classical knot $K =6.73053,$ which has
\begin{eqnarray*}
G_K &=&  \langle a, b \mid a^{-1}b^{-1}a^{-1}bab \rangle \\
G_{K^*} &=& \langle  a, b \mid a^{-1}b^{-1}a^{-1}bab, a^3b^{-3}aba^{-2}b^2a^{-1}b^{-1}\rangle.
\end{eqnarray*}
Notice that the first relation for these groups is identical, and by counting representations into the symmetric group $S_4$ on four letters, one can show that $G_K \not\cong G_{K^*}$.
There are several other almost classical knots with $G_K \not\cong G_{K^*}$, and they include the virtual knots $6.76479, \, 6.78358, \, 6.87188, \, 6.89815$ and $6.90115$. In all cases, the groups $G_K$ and $G_{K^*}$ can be distinguished by counting representations into $S_4$ or $S_5.$

\begin{table}[h]
\begin{center}
\begin{tabular}{c| r r r r r r r}
\hline
Crossing  \; & virtual & mod $2$ & mod $3$ & mod $4$ & mod $5$ & almost & classical\\
number & knots &AC&AC&AC&AC&classical & knots \\
\hline
2  & 1         & 0       & 0      & 0   &   0  & 0 & 0  \\
3  & 7         & 3       & 1      & 1   &   1  & 1 & 1\\
4  & 108     & 10     & 6      & 3   & 3    & 3 & 1\\
5  & 2448   & 104   & 21    & 17 & 11  & 11 & 2 \\
6  & 90235 & 1557 & 192  & 81 & 71  & 61  & 5\\  
\end{tabular}
\end{center}
\bigskip
\caption{The numbers of mod $p$ almost classical knots up to $6$ crossings.}
\label{Numbers-AC}
\end{table}

%%%%%%%%%%%%%%%%%%%%%%%%%%%%%%%%%%%%%%%%%%%%%%%%%%%%%%%%%%%%%%%%%%%%%%%%%
\begin{table}[ht] %\footnotesize

\begin{tabular}{cc}
\begin{minipage}{0.49\textwidth}
\begin{tabular}{|l|c|c|}
\hline
Knot & $g$ & Alexander polynomial  \\
\hline \hline
${3.6}$ &0& $t^2-t+1$\\ \hline
4.99  &1& $2t-1$ \\ \hline
4.105 &1& $2t^2-2t+1$ \\ \hline
${4.108}$  &0& $t^2-3t+1 $ \\  \hline
5.2012 &2& $1$ \\  \hline
5.2025 &2& $1 $ \\  \hline
5.2080 &1& $1$ \\  \hline
5.2133 &2& $2t-1$ \\  \hline
5.2160 &1& $t^2-t+1 $ \\  \hline
5.2331 &1& $t^3-t+1 $ \\  \hline
5.2426 &1& $(t^2-t+1)^2$ \\  \hline
5.2433 &2& $t^4-2t^3 +4t^2 -3t +1 $ \\  \hline
${5.2437}$ &0& $2t^2 - 3t + 2 $ \\  \hline
5.2439 &1& $t^3-2t^2 +3t -1$ \\  \hline
${5.2445}$ &0& $ t^4 - t^3 + t^2 - t + 1$\\  \hline
6.72507 &2& $1$ \\  \hline
6.72557 &2&$1$ \\  \hline
6.72692 &2& $1 $ \\  \hline
6.72695 &1& $1$ \\  \hline
6.72938 &2& $t^2-t+1  $ \\  \hline
6.72944 &1& $2t-1  $ \\  \hline
6.72975 &2& $1$ \\  \hline
6.73007 &2& $1$ \\  \hline
6.73053 &1& $t^2-t+1 $ \\  \hline
6.73583 &2&$1$ \\  \hline
6.75341 &1& $2t-1 $ \\  \hline
6.75348 &2& $t-2$ \\  \hline
6.76479 &2& $t^2-t+1 $ \\  \hline
6.77833 &2& $t^2-t+1 $ \\  \hline
6.77844 &2& $t^2-t+1 $ \\  \hline
6.77905 &2& $t^2-3t+1$ \\  \hline
6.77908 &1& $2t^2-2t+1 $ \\  \hline
6.77985 &2& $t^2-t+1 $ \\  \hline
6.78358 &2& $t^2-3t+1 $ \\  \hline
6.79342 &1& $t^2-3t+1$ \\  \hline
6.85091 &1& $t^2+t-1 $ \\  \hline
6.85103 &1& $t^3-t^2+2t-1 $ \\  \hline
6.85613\; &2& $t^2-2t+2 $ \\  \hline

\end{tabular}
\end{minipage}

\begin{minipage}{0.49\textwidth}
 \begin{tabular}{|l|c|c|} 
\hline
Knot &  $g$  & Alexander polynomial   \\ 
\hline \hline
6.85774 &2& $t^3 - t^2 + 1 $ \\  \hline
6.87188 &1& $(2t - 1)(t^2 - t + 1) $ \\  \hline
6.87262 &2& $(t^2-t+1)^2$ \\  \hline
6.87269 &1& $(2t-1)^2$ \\  \hline
6.87310 &1& $t^4 -t^3 +2t^2 -2t +1 $ \\  \hline
6.87319 &2& $3t^2-3t+1 $ \\  \hline
6.87369 &2& $t^3-2t^2+3t-1 $ \\  \hline
6.87548 &1& $t^3-2t^2 -t+1 $ \\  \hline
6.87846 &1& $t^3-2t^2-t+1  $ \\  \hline
6.87857 &1& $t^2-4t+2 $ \\  \hline
6.87859 &1& $3t^2-3t+1 $ \\  \hline
6.87875 &1& $t^3+t^2-2t+1 $ \\  \hline
6.89156 &1&  $2t^3-t^2-t+1 $ \\  \hline
6.89187 &0& $(t^2-t+1)^2$ \\  \hline
6.89198 &0& $(t^2-t+1)^2$ \\  \hline
6.89623 &1& $2t^2-2t+1 $ \\  \hline
6.89812 &1& $t^3 -2t+2 $ \\  \hline
6.89815 &1& $2t - 1 $ \\  \hline
6.90099 &2& $t^4 - t^2 + 1 $ \\  \hline
6.90109 &1& $(2t^2-2t+1)(t^2-t+1)$ \\  \hline
6.90115 &1& $(t^2 - 3t + 1)(t^2 - t + 1)$ \\  \hline
6.90139 &1& $(3t^2 - 3t + 1)(t^2 - t + 1)$ \\  \hline
6.90146 &2& $t^4 - 5t^3 + 9t^2 - 5t + 1$ \\  \hline
6.90147 &2& $t^4-3t^3+6t^2-5t+2 $ \\  \hline
6.90150 &1& $t^4 -5t^3+6t^2-4t+1 $ \\  \hline
6.90167 &1& $t^4-2t^3+4t^2 -4t+2 $ \\  \hline
6.90172 &0& $t^4-3t^3+5t^2-3t+1 $ \\  \hline
6.90185 &2& $3t^4-6t^3+6t^2-3t+1$ \\  \hline
6.90194 &1& $t^4-4t^3+8t^2-5t+1$ \\  \hline
6.90195 &1& $2t^4-3t^3+3t^2-2t+1 $ \\  \hline
6.90209 &0& $t^4-3t^3+3t^2-3t+1 $ \\  \hline
6.90214 &1& $3t^2-4t+2 $ \\  \hline
6.90217 &1& $t^3-4t^2+3t-1 $ \\  \hline
6.90219 &1& $2t^3 - 3t^2 + 3t - 1$ \\  \hline
6.90227 &0& $(t - 2)(2t - 1)$ \\  \hline
6.90228 &1& $4t^2 - 6t + 3$ \\  \hline
6.90232 &1& $2 t^3 - 6 t^2 + 4t - 1$ \\  \hline
6.90235 \; &2& $t^4 -3t^3+5t^2-3t+1 $ \\  \hline  
\end{tabular}
\end{minipage}
\end{tabular}
\vspace{5mm}
\caption{The virtual genus $g=g(\Si_K)$ and Alexander polynomial for almost classical knots with up to six crossings.}
\label{acks2}
\end{table}

\noindent
{\bf Concluding remarks.} In closing, we remark that one can develop similar results for long virtual knots. For instance, given a long virtual knot, there are natural definitions for the virtual knot group $\VG_K$, the reduced virtual knot group $\RG_K$, the extended group $\EG_K$, and the quandle groups $\QG_K$, and one can prove analogues to Theorems \ref{vgk-amalgamation} and \ref{Duality} for long virtual knots. 

The corresponding theory of Alexander invariants is actually much simpler because the first elementary ideals associated to each of these groups are in general principal for long virtual knots. Particularly nice is the case of almost classical long knots as many of the invariants of classical knots extend in the most natural way. This follows by viewing virtual knots as knots in thickened surfaces and interpreting long virtual knots as knots in thickened surfaces with a choice of basepoint. Drilling a small vertical hole through the thickened surface $\Sigma \times I$ at the basepoint results in 3-manifold $M$ with $\partial M$ connected. In particular, $M$ is a quasi-cylinder with $H_2(M) =0,$ and this is precisely the condition required to extend classical knot invariants to quasi-cylinders in the paper \cite{CT} by Cimasoni and Turaev. For instance, by appealing to the results in \cite{CT}, one can define a Conway normalized Alexander polynomial, knot signatures, etc. We plan to address these questions in future work, where we hope to develop the corresponding results on virtual knot groups and almost classicality for long virtual knots.

\subsection*{Acknowledgements} We extend our thanks to Dror Bar-Natan, Ester Dalvit, Emily Dies, Jim Hoste and Vassily Manturov for their valuable input.  H. Boden and A. Nicas were supported by grants from the Natural Sciences and Engineering Research Council of Canada, and
R. Gaudreau was supported by a Postgraduate Scholarship from  the Natural Sciences and Engineering Research Council of Canada.

%\appendix{A tabulation of almost classical knots up to six crossings}

\begin{figure}[h]
%\centering
\tiny{
${3.6=3_1}$ \hspace{-0.8cm} \includegraphics[scale=0.65]{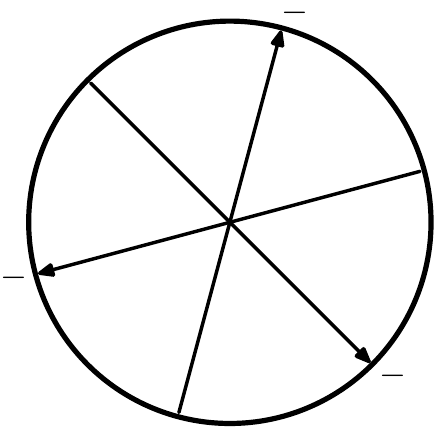} \hspace{0.8cm}
4.99 \hspace{-0.6cm}\includegraphics[scale=0.65]{GD-4-99.pdf} \hspace{0.8cm}  
4.105 \hspace{-0.6cm}\includegraphics[scale=0.65]{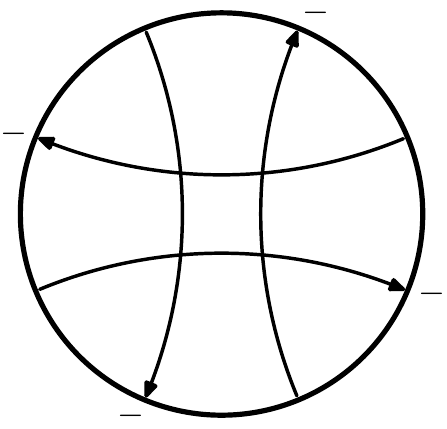}  

\smallskip
${4.108=4_1}$ \hspace{-0.8cm} \includegraphics[scale=0.65]{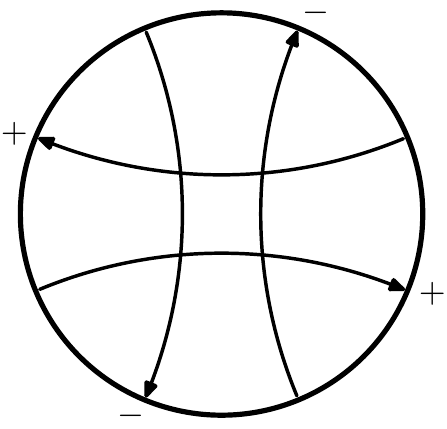} \hspace{0.6cm}
5.2012 \hspace{-0.6cm}\includegraphics[scale=0.65]{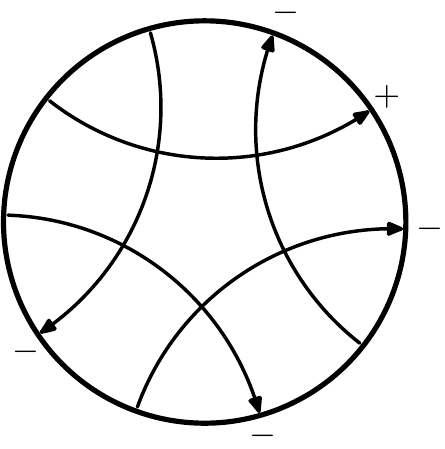}  \hspace{0.6cm}
5.2025 \hspace{-0.6cm}\includegraphics[scale=0.65]{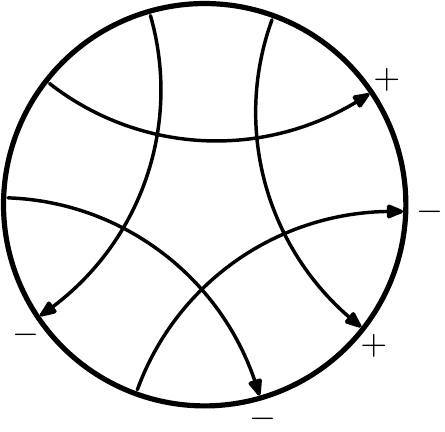}  

\smallskip
5.2080 \hspace{-0.6cm}\includegraphics[scale=0.65]{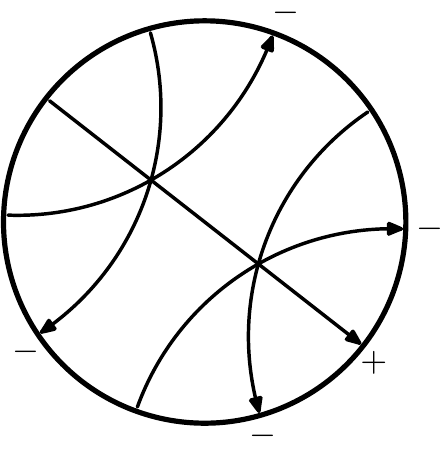} \hspace{0.8cm}
5.2133 \hspace{-0.6cm}\includegraphics[scale=0.65]{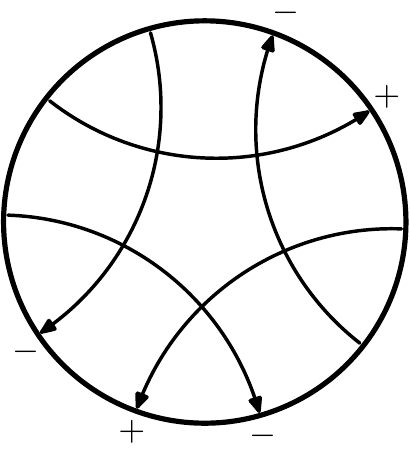} \hspace{0.8cm}
5.2160 \hspace{-0.6cm}\includegraphics[scale=0.65]{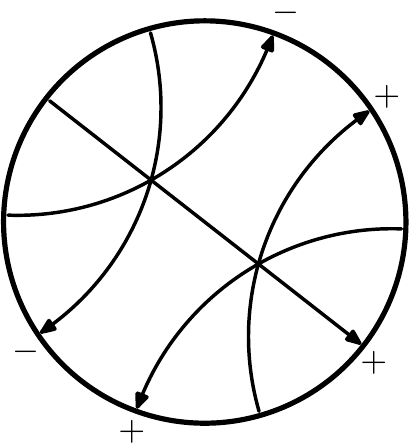}  

\smallskip
5.2331 \hspace{-0.6cm}\includegraphics[scale=0.65]{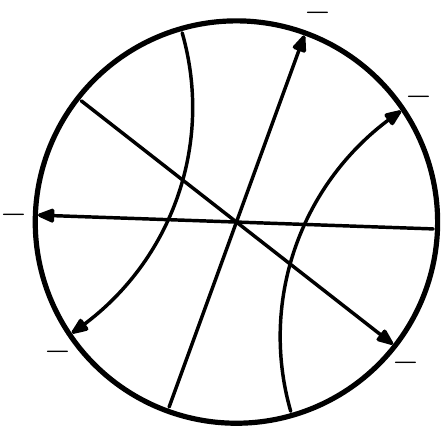} \hspace{0.8cm}
5.2426 \hspace{-0.6cm}\includegraphics[scale=0.65]{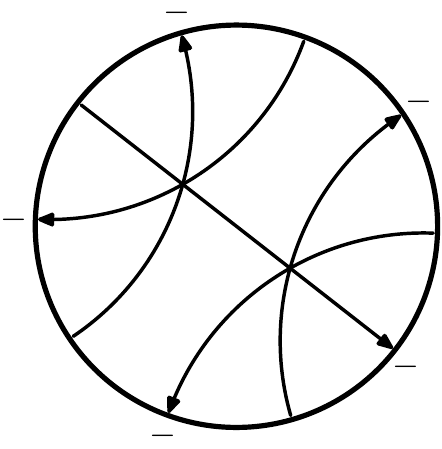} \hspace{0.8cm}
5.2433 \hspace{-0.6cm}\includegraphics[scale=0.65]{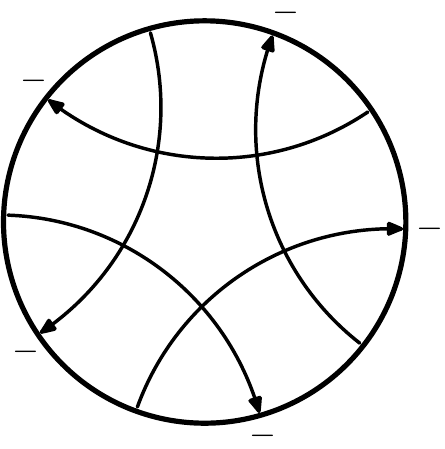}  

\smallskip
$5.2437=5_2$ \hspace{-1.0cm}\includegraphics[scale=0.65]{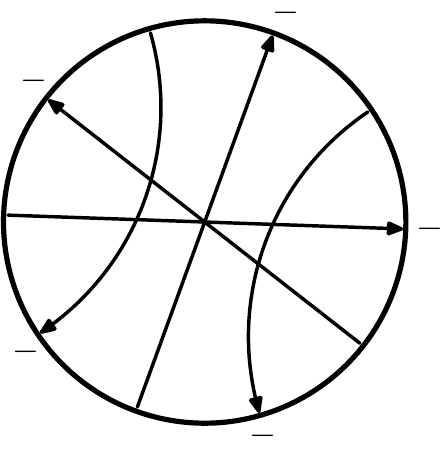} \hspace{0.8cm}
5.2439 \hspace{-0.6cm}\includegraphics[scale=0.65]{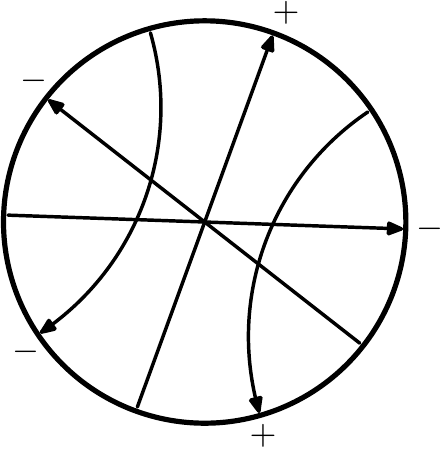} \hspace{0.1cm}
$5.2445=5_1$ \hspace{-1.0cm}\includegraphics[scale=0.65]{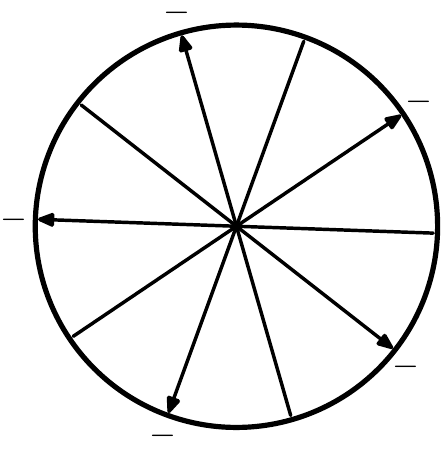} 

\smallskip 
6.72507 \hspace{-0.7cm}\includegraphics[scale=0.65]{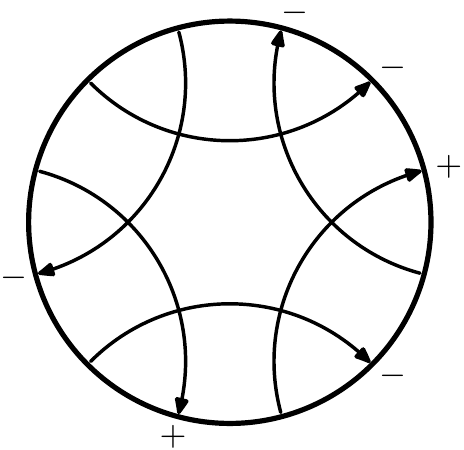} \hspace{0.7cm} 
6.72557 \hspace{-0.7cm}\includegraphics[scale=0.65]{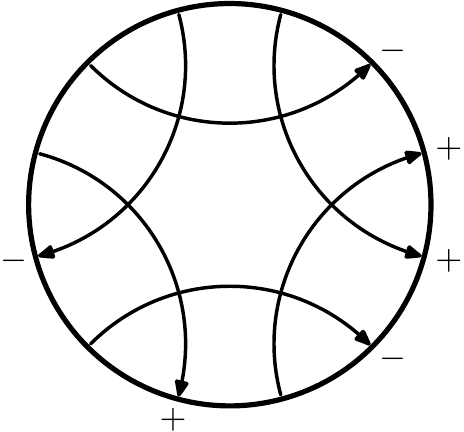} \hspace{0.7cm}
6.72692 \hspace{-0.7cm}\includegraphics[scale=0.65]{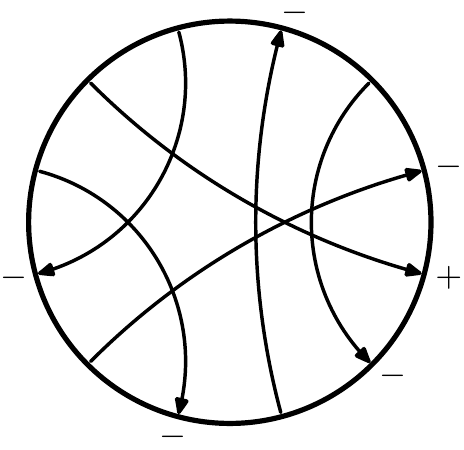} 

\smallskip
6.72695 \hspace{-0.7cm}\includegraphics[scale=0.65]{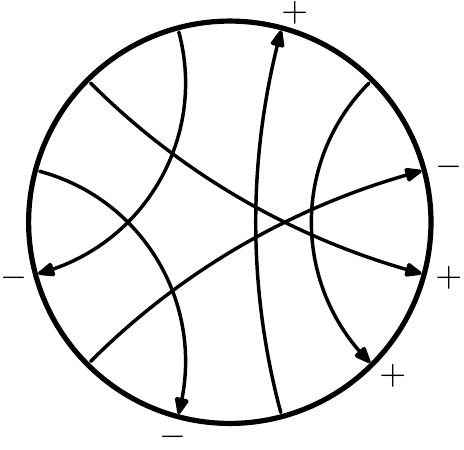} \hspace{0.7cm} 
6.72938 \hspace{-0.7cm}\includegraphics[scale=0.65]{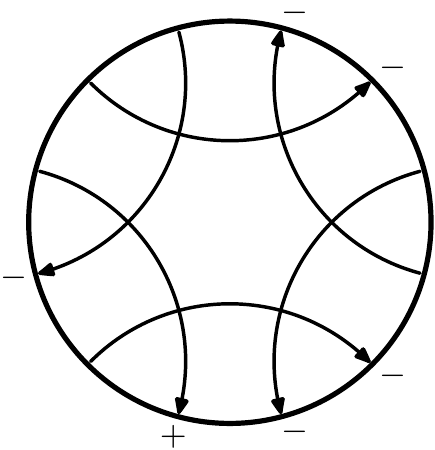} \hspace{0.7cm} 
6.72944 \hspace{-0.7cm}\includegraphics[scale=0.65]{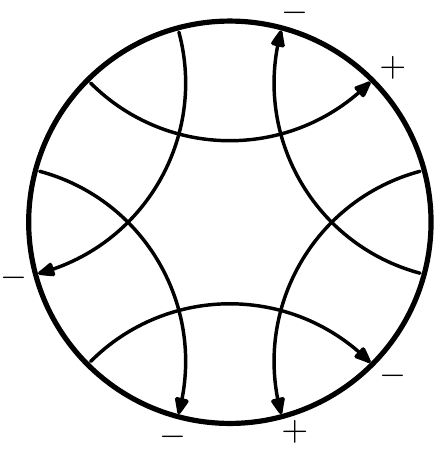} }
\end{figure} 

\begin{figure}[ht]
%\centering
\tiny{
6.72975 \hspace{-0.7cm}\includegraphics[scale=0.65]{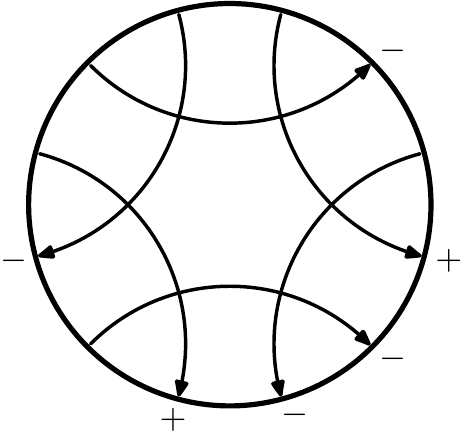} \hspace{0.7cm}
6.73007 \hspace{-0.7cm}\includegraphics[scale=0.65]{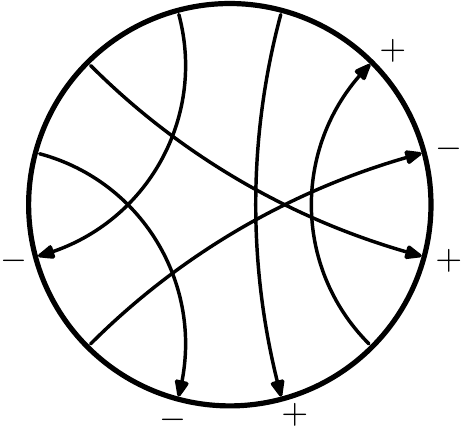} \hspace{0.7cm}
6.73053 \hspace{-0.7cm}\includegraphics[scale=0.65]{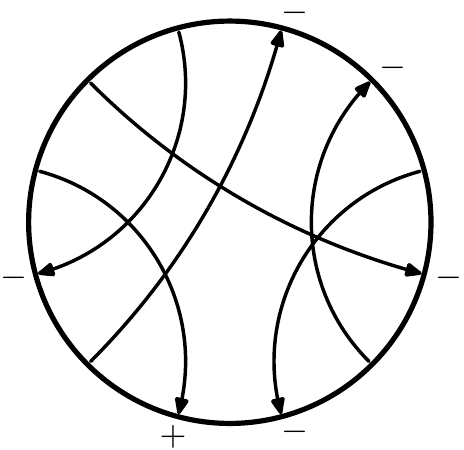} 

\smallskip
6.73583 \hspace{-0.7cm}\includegraphics[scale=0.65]{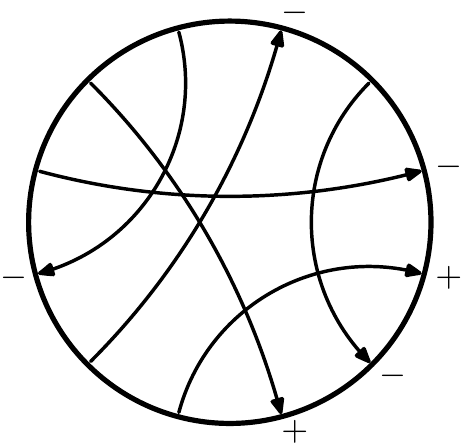} \hspace{0.7cm}
6.75341 \hspace{-0.7cm}\includegraphics[scale=0.65]{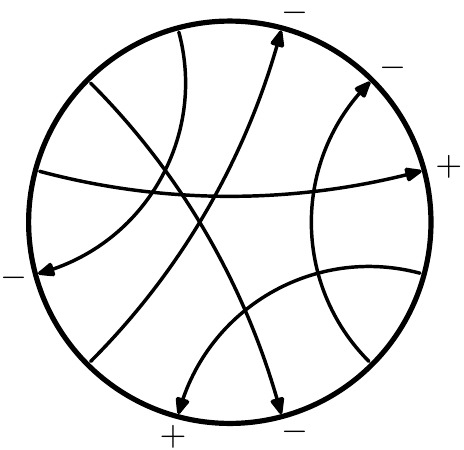} \hspace{0.7cm}
6.75348 \hspace{-0.7cm}\includegraphics[scale=0.65]{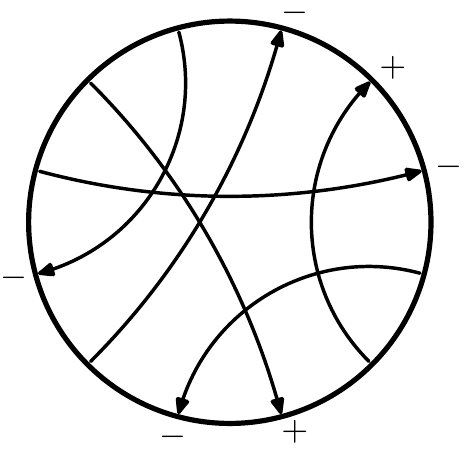}

\smallskip
6.76479 \hspace{-0.7cm}\includegraphics[scale=0.65]{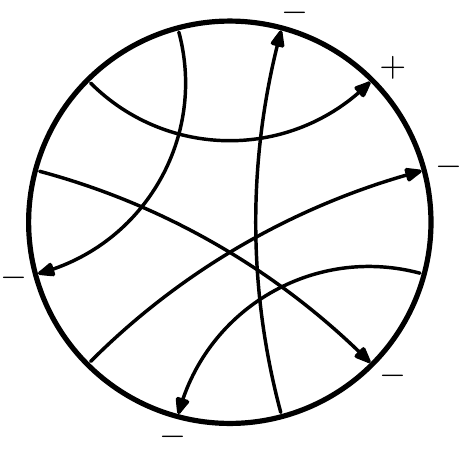} \hspace{0.7cm}
6.77833 \hspace{-0.7cm}\includegraphics[scale=0.65]{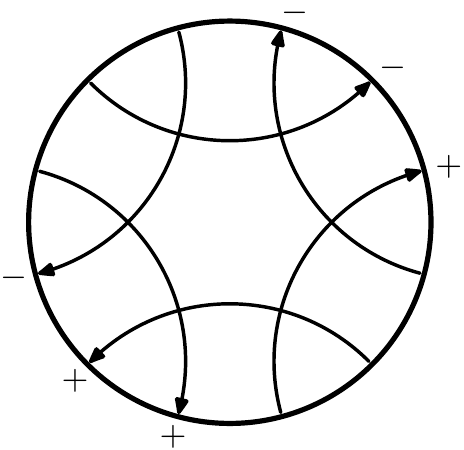} \hspace{0.7cm}
6.77844 \hspace{-0.7cm}\includegraphics[scale=0.65]{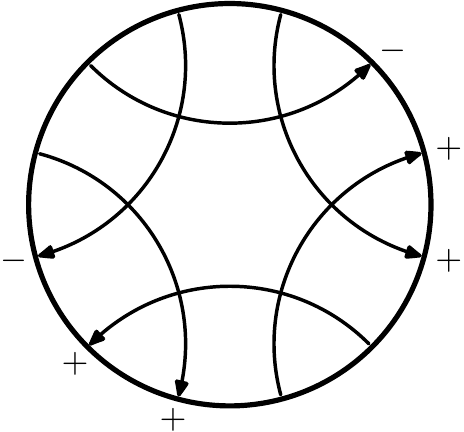} 

\smallskip 
6.77905 \hspace{-0.7cm}\includegraphics[scale=0.65]{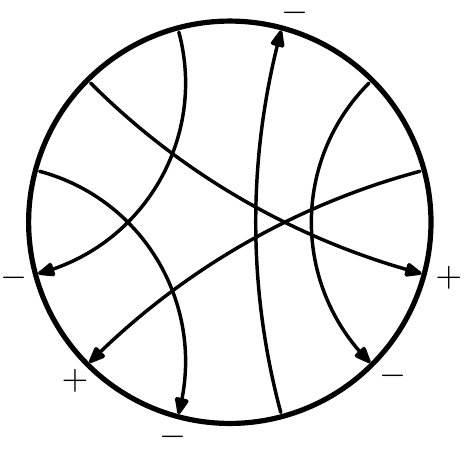} \hspace{0.7cm}
6.77908 \hspace{-0.7cm}\includegraphics[scale=0.65]{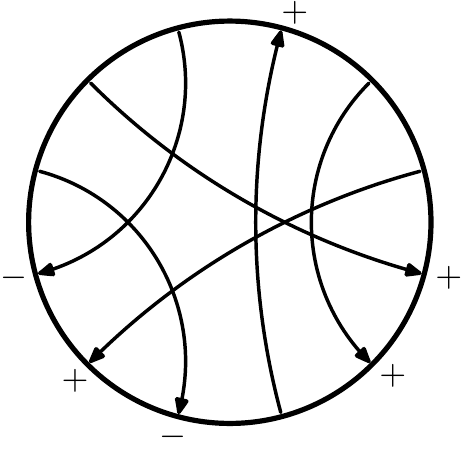} \hspace{0.7cm}
6.77985 \hspace{-0.7cm}\includegraphics[scale=0.65]{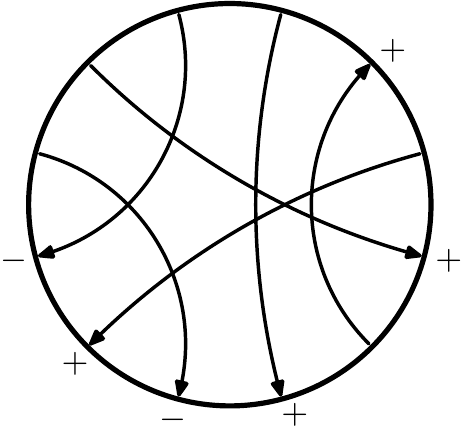} 

\smallskip 
6.78358 \hspace{-0.7cm}\includegraphics[scale=0.65]{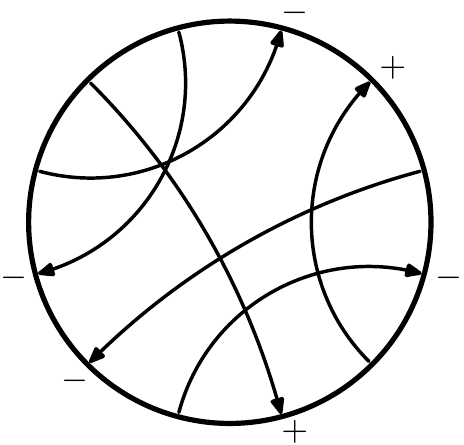} \hspace{0.7cm}
6.79342 \hspace{-0.5cm}\includegraphics[scale=0.65]{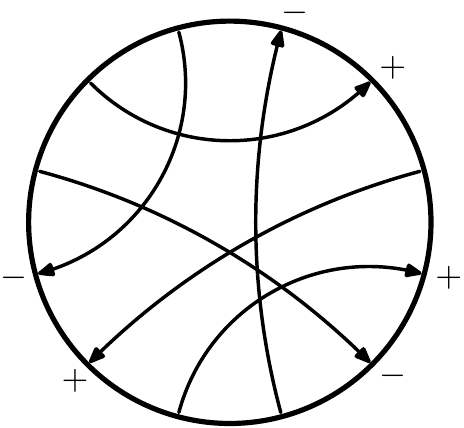} \hspace{0.7cm}
6.85091 \hspace{-0.7cm}\includegraphics[scale=0.65]{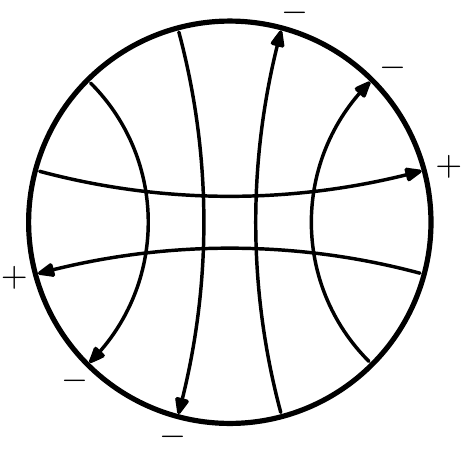}  
 
\smallskip
6.85103 \hspace{-0.7cm}\includegraphics[scale=0.65]{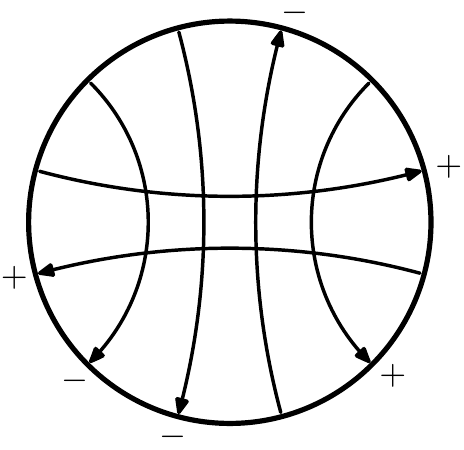} \hspace{0.7cm}
6.85613 \hspace{-0.5cm}\includegraphics[scale=0.65]{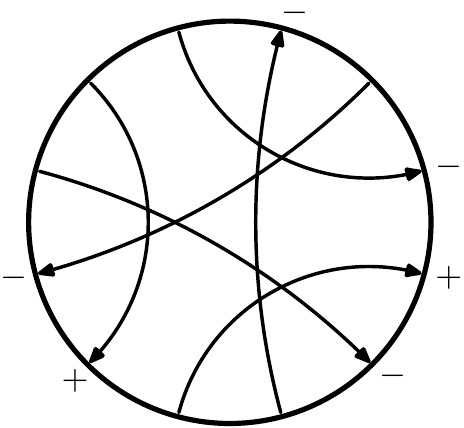} \hspace{0.7cm}
6.85774 \hspace{-0.7cm}\includegraphics[scale=0.65]{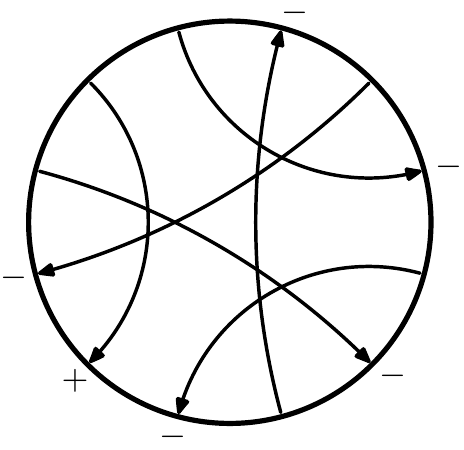} 

\smallskip
6.87188 \hspace{-0.7cm}\includegraphics[scale=0.65]{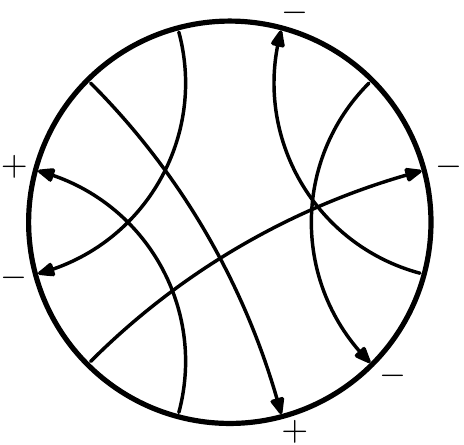} \hspace{0.7cm}
6.87262 \hspace{-0.7cm}\includegraphics[scale=0.65]{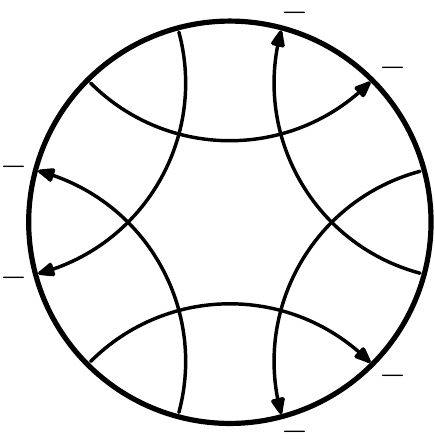} \hspace{0.7cm}
6.87269 \hspace{-0.7cm}\includegraphics[scale=0.65]{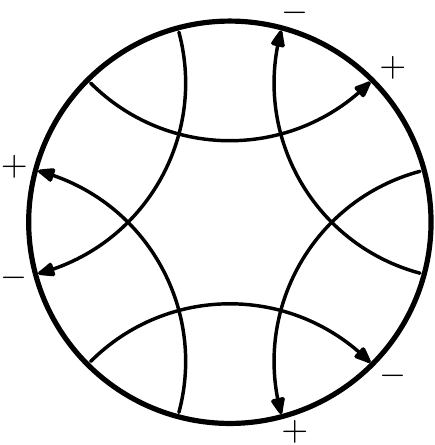} }
\end{figure} 

\begin{figure}[h]
%\centering
\tiny{
6.87310 \hspace{-0.7cm}\includegraphics[scale=0.65]{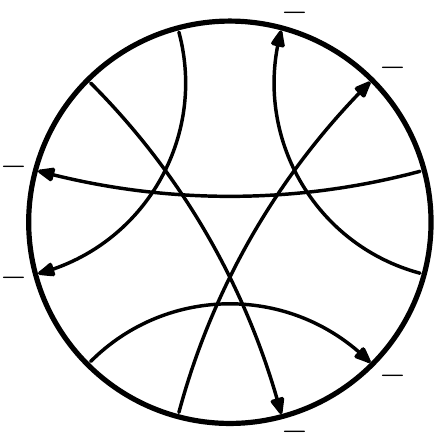} \hspace{0.7cm}
6.87319 \hspace{-0.7cm}\includegraphics[scale=0.65]{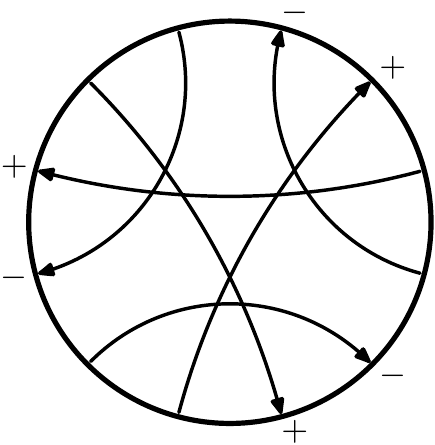} \hspace{0.7cm}
6.87369 \hspace{-0.7cm}\includegraphics[scale=0.65]{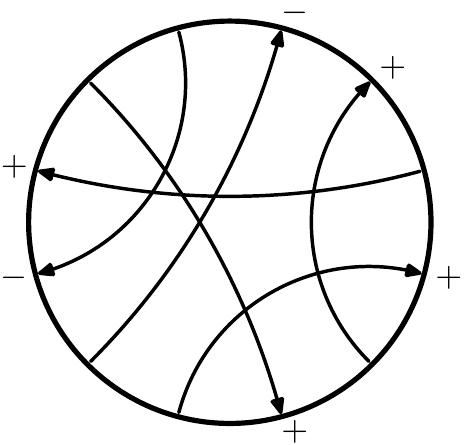}  

\smallskip
6.87548 \hspace{-0.7cm}\includegraphics[scale=0.65]{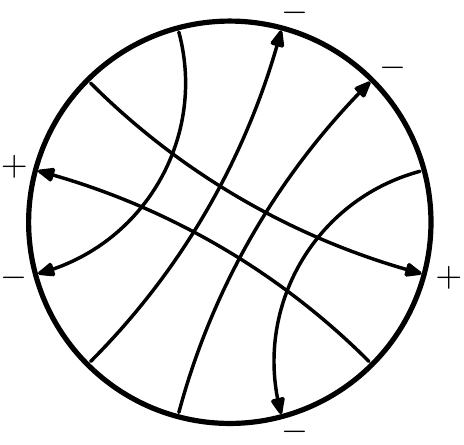} \hspace{0.7cm}
6.87846 \hspace{-0.7cm}\includegraphics[scale=0.65]{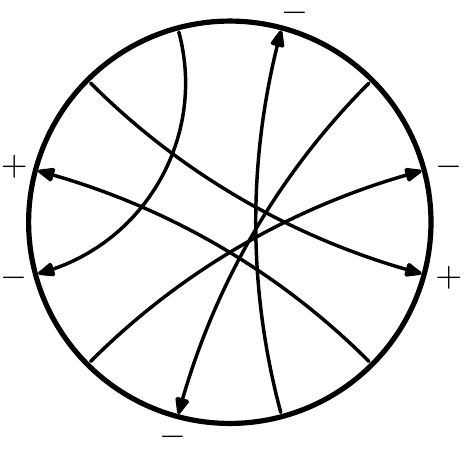} \hspace{0.7cm}
6.87857 \hspace{-0.7cm}\includegraphics[scale=0.65]{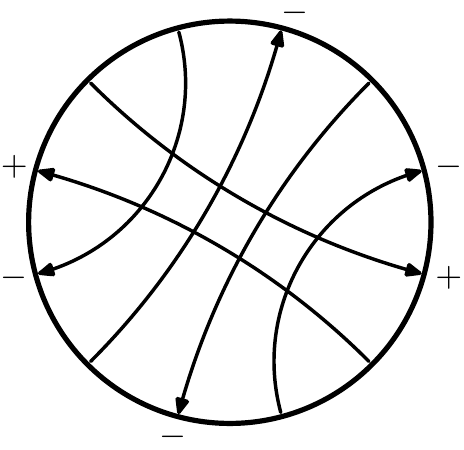}

\smallskip
6.87859 \hspace{-0.7cm}\includegraphics[scale=0.65]{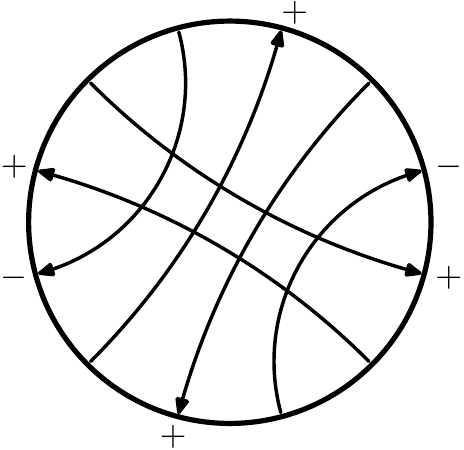} \hspace{0.7cm}
6.87875 \hspace{-0.7cm}\includegraphics[scale=0.65]{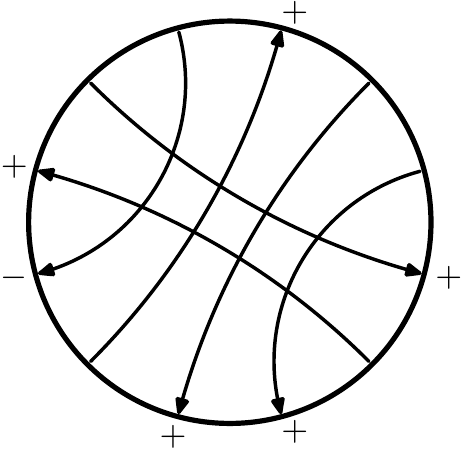} \hspace{0.7cm}
6.89156 \hspace{-0.9cm}\includegraphics[scale=0.65]{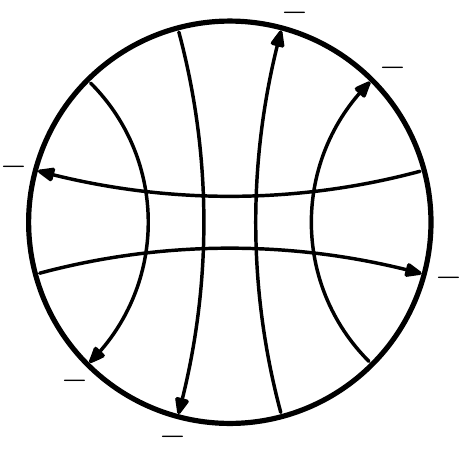}  

\smallskip
$\stackrel{\hbox{$6.89187\quad $}}{\hbox{$=3_1 \# 3_1$}}$
\hspace{-0.7cm}\includegraphics[scale=0.65]{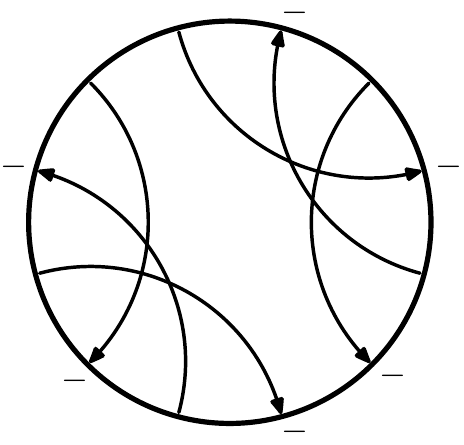}  \hspace{0.1cm}
$\stackrel{\hbox{$6.89198\quad $}}{\hbox{$=3_1 \# 3_1^*$}}$
\hspace{-0.7cm}\includegraphics[scale=0.65]{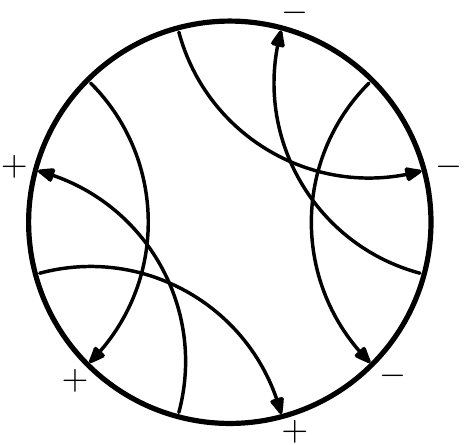} \hspace{0.5cm}
6.89623 \hspace{-0.7cm}\includegraphics[scale=0.65]{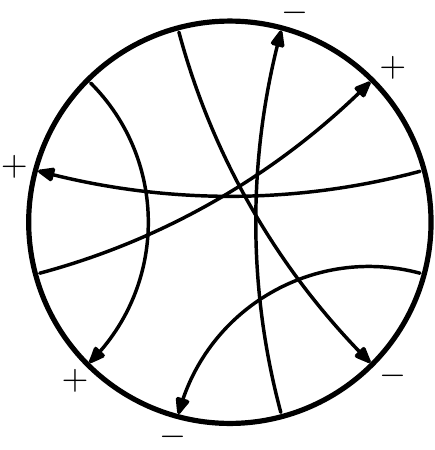}  

\smallskip
6.89812 \hspace{-0.7cm}\includegraphics[scale=0.65]{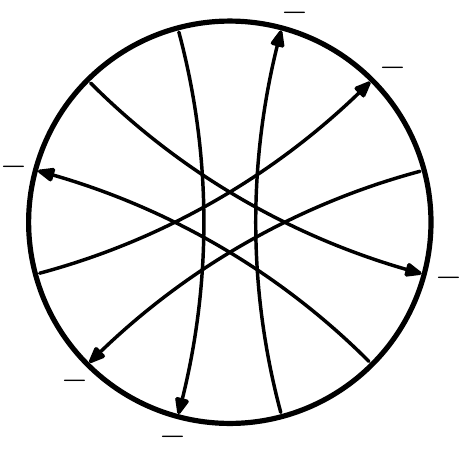}  \hspace{0.7cm}
6.89815 \hspace{-0.7cm}\includegraphics[scale=0.65]{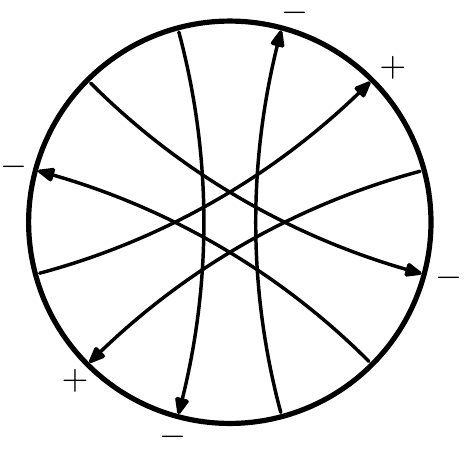} \hspace{0.7cm} 
6.90099 \hspace{-0.7cm}\includegraphics[scale=0.65]{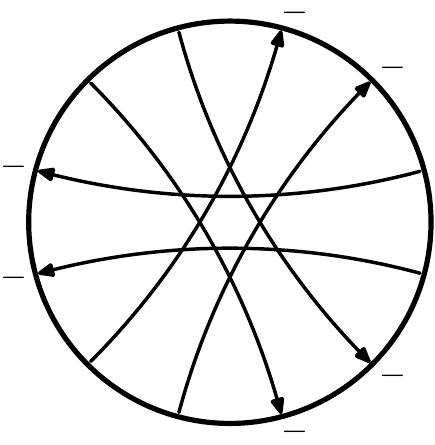}

\smallskip
6.90109 \hspace{-0.7cm}\includegraphics[scale=0.65]{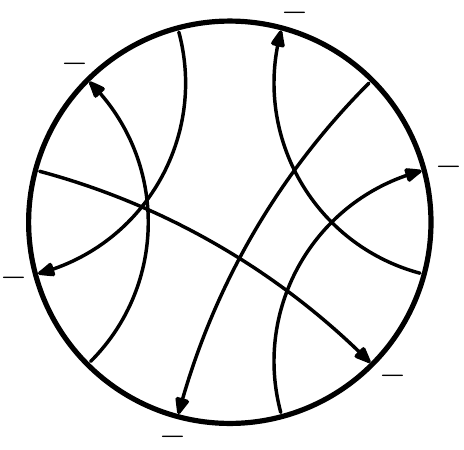} \hspace{0.7cm}
6.90115 \hspace{-0.7cm}\includegraphics[scale=0.65]{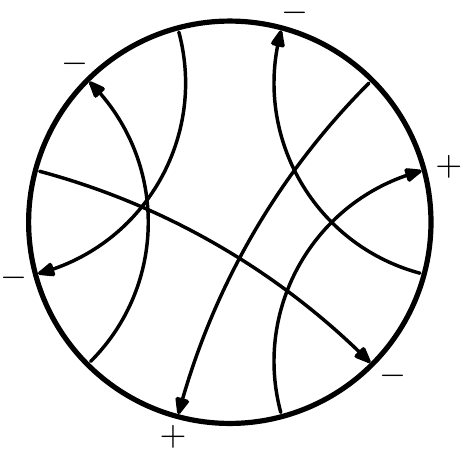} \hspace{0.7cm}
6.90139 \hspace{-0.7cm}\includegraphics[scale=0.65]{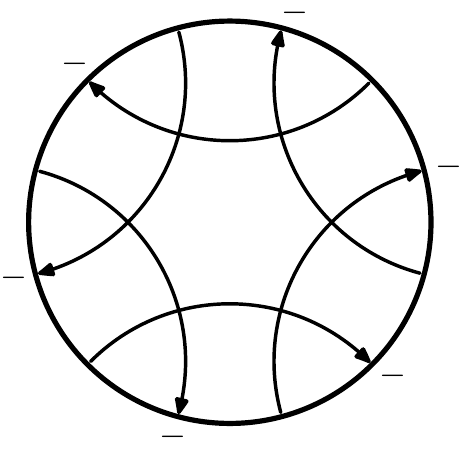} 

\smallskip
6.90146 \hspace{-0.7cm}\includegraphics[scale=0.65]{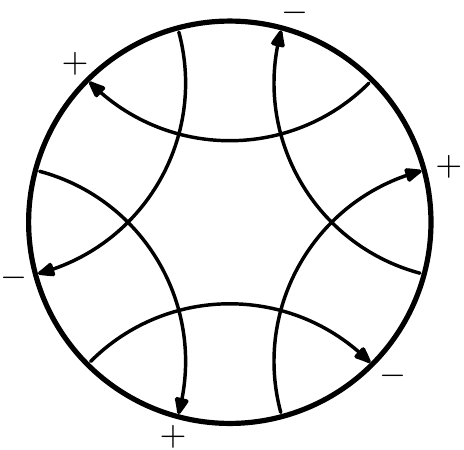} \hspace{0.7cm}
6.90147 \hspace{-0.7cm}\includegraphics[scale=0.65]{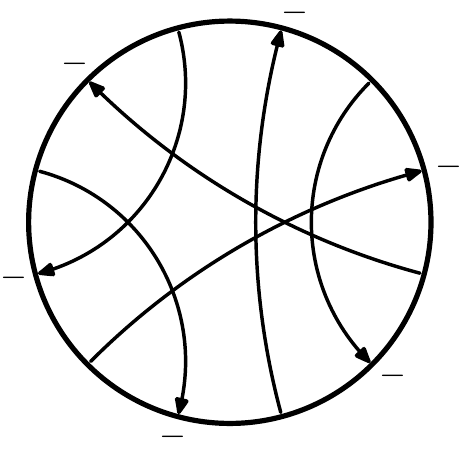} \hspace{0.7cm}
6.90150 \hspace{-0.7cm}\includegraphics[scale=0.65]{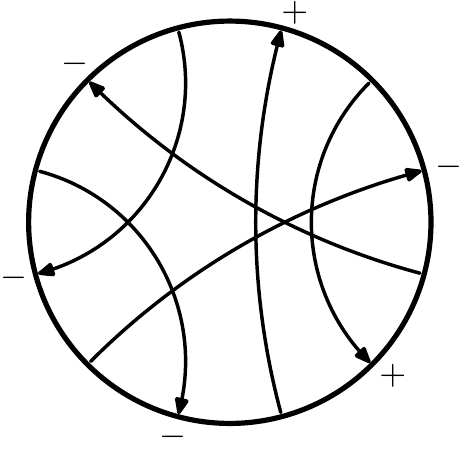}   
 }
\end{figure}

\begin{figure}[h]
%\centering
\tiny{
6.90167 \hspace{-0.7cm}\includegraphics[scale=0.65]{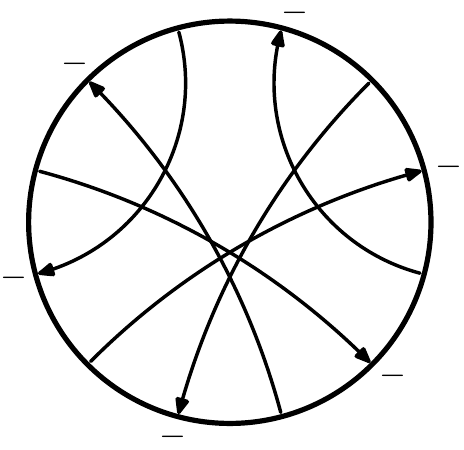}  \hspace{0.5cm}
$6.90172=6_3$ \hspace{-0.9cm}\includegraphics[scale=0.65]{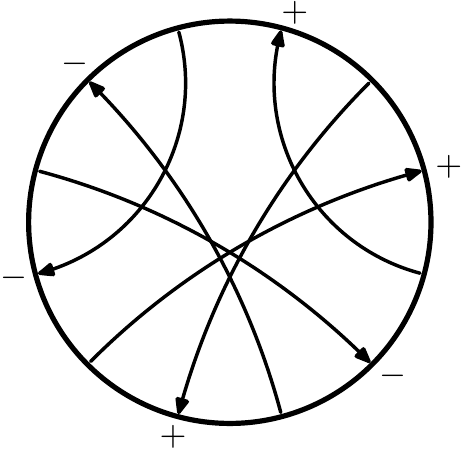}  \hspace{0.5cm}
6.90185 \hspace{-0.7cm}\includegraphics[scale=0.65]{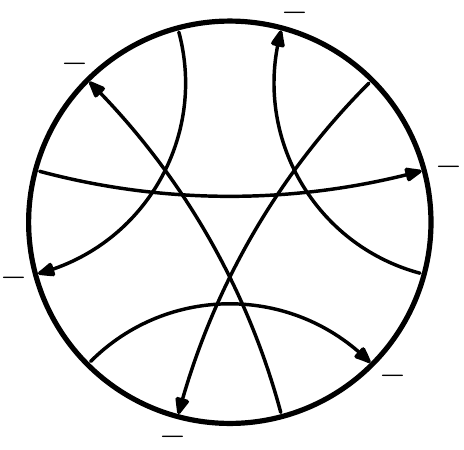}  

\smallskip
6.90194 \hspace{-0.7cm}\includegraphics[scale=0.65]{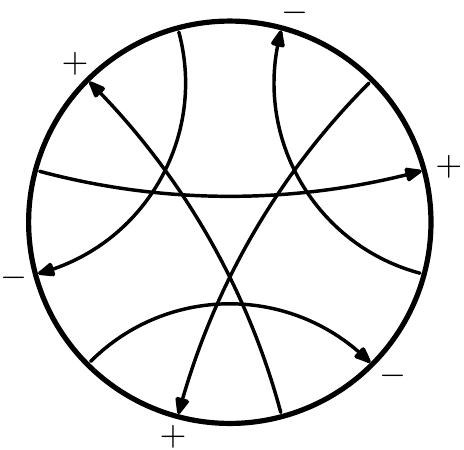} \hspace{0.5cm}
6.90195 \hspace{-0.7cm}\includegraphics[scale=0.65]{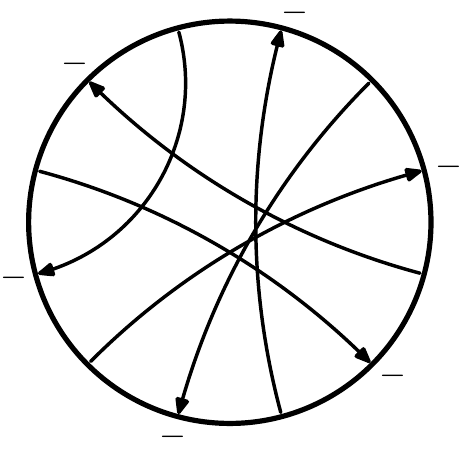} \hspace{0.5cm}
$6.90209=6_2$ \hspace{-1cm}\includegraphics[scale=0.65]{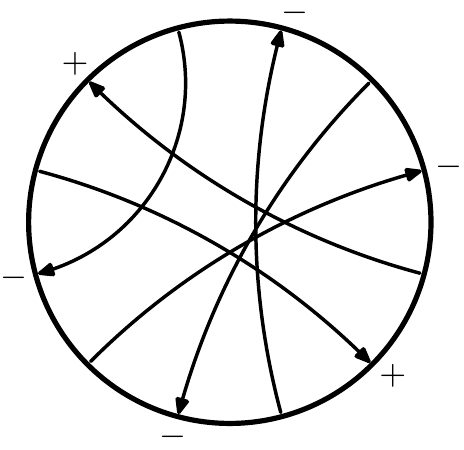} 

\smallskip
6.90214 \hspace{-0.7cm}\includegraphics[scale=0.65]{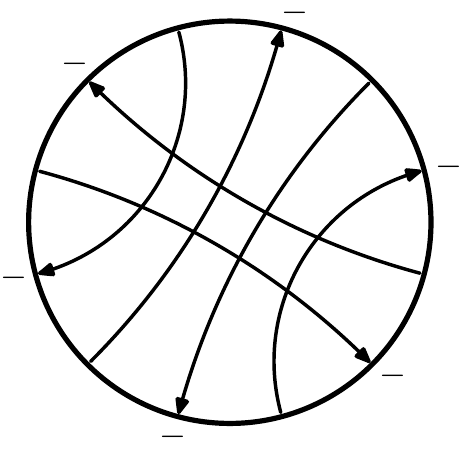} \hspace{0.5cm}
6.90217 \hspace{-0.7cm}\includegraphics[scale=0.65]{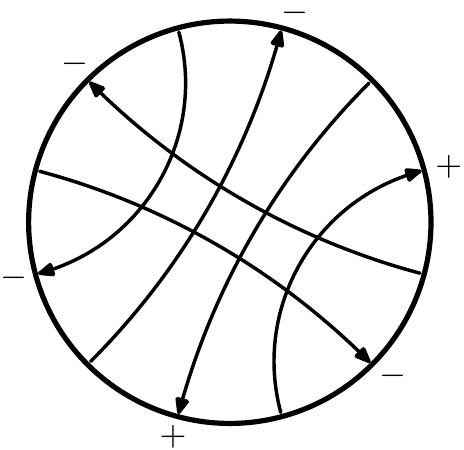} \hspace{0.5cm}
6.90219 \hspace{-0.7cm}\includegraphics[scale=0.65]{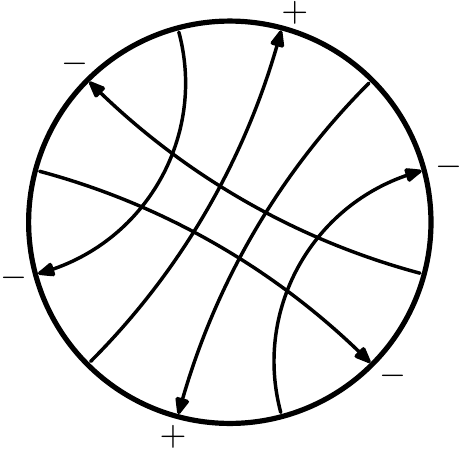} 

\smallskip
$6.90227=6_1$ \hspace{-1cm}\includegraphics[scale=0.65]{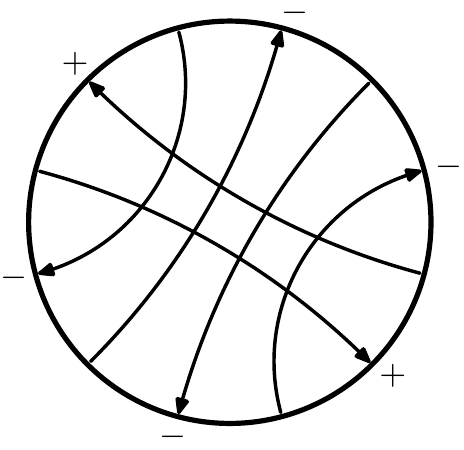} \hspace{0.5cm}
6.90228 \hspace{-0.7cm}\includegraphics[scale=0.65]{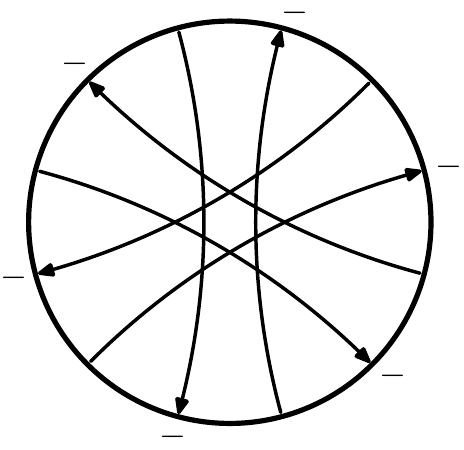} 

\smallskip
6.90232 \hspace{-0.7cm}\includegraphics[scale=0.65]{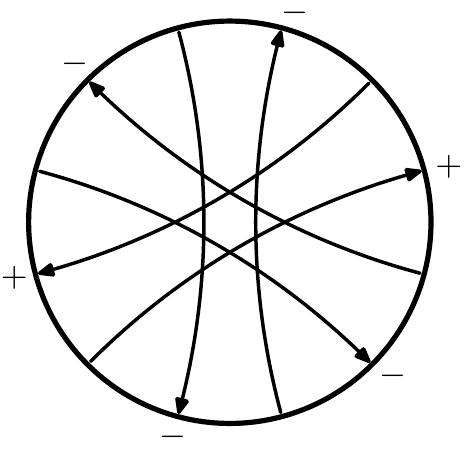} \hspace{0.5cm}
6.90235 \hspace{-0.7cm}\includegraphics[scale=0.65]{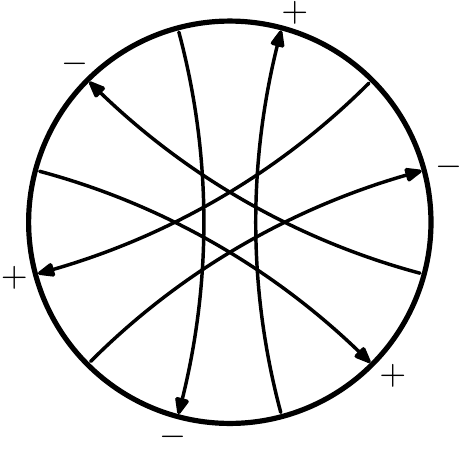}
}
\vspace{5mm}
\caption{ 76 Gauss diagrams of almost classical knots with up to six crossings, including the nine classical knots $3_1, 4_1, 5_1, 5_2, 6_1, 6_2, 6_3, 3_1 \# 3_1,$ and $3_1 \# 3_1^*$ as indicated.}
\label{AC-knots}
\end{figure}

\end{document}